\documentclass[10pt]{article}
 \usepackage[dvips]{graphicx}
 \usepackage{psfrag}
 \let\theoremstyle\relax
 \usepackage{amsfonts, amssymb, amsmath, amsthm}
 \usepackage{color}
 \usepackage{mathabx}
\usepackage{euscript}
\usepackage{mathrsfs}
\usepackage{appendix}
\usepackage{enumitem}
\usepackage{lscape}
\usepackage{multirow}
\usepackage[table,xcdraw]{xcolor}
\usepackage{longtable}
\usepackage{graphicx, array, blindtext}
\usepackage[table,xcdraw]{xcolor}
\usepackage[space]{grffile}
\usepackage{float}
\usepackage{caption}
\usepackage{subcaption}
\newcolumntype{P}[1]{>{\centering\arraybackslash}p{#1}}
\DeclareMathAlphabet{\mathpzc}{OT1}{pzc}{m}{it}

\renewcommand{\a}{\alpha}   \renewcommand{\b}{\beta}

\renewcommand{\d}{\delta}   \newcommand{\e}{\epsilon}
\newcommand{\ve}{\varepsilon}
      
   \renewcommand{\l}{\lambda}
\renewcommand{\L}{\Lambda}

\newcommand{\p}{\phi}
\newcommand{\vp}{\varphi}    
\renewcommand{\r}{\rho}      \newcommand{\s}{\sigma}
   \renewcommand{\t}{\tau}

   \newcommand{\G}{\Gamma}
\newcommand{\D}{\Delta}

\newcommand{\newvec}[1]{\vbox{\ialign{##\crcr$\displaystyle\rightharpoonup$\crcr\noalign{\kern-1pt\nointerlineskip}
$\hfil\displaystyle{#1}\hfil$\crcr}}} 
\newcommand{\longrightharpoonup}{\relbar\joinrel\rightharpoonup}
\newcommand{\longvec}[1]{\vbox{\ialign{##\crcr$\displaystyle\longrightharpoonup$\crcr\noalign{\kern-1pt\nointerlineskip}
$\hfil\displaystyle{#1}\hfil$\crcr}}} 

\newcommand{\RR}{{\mathbb R}}
\newcommand{\CC}{{\mathbb C}}
\newcommand{\NN}{{\mathbb N}}

\newcommand{\tv}{\text{\bf v}}
\newcommand{\tn}{\text{\bf n}}
\newcommand{\tp}{\text{\bf p}}
\newcommand{\tf}{\text{\bf f}}
\newcommand{\tx}{\text{\bf x}}
\newcommand{\sgn}{\text{sgn}}
\newcommand{\re}{\text{Re}}

\newcommand{\dom}{\text{Dom}}
\newcommand{\res}{\text{Res}}
\newcommand{\disc}{\text{disc}}

\newtheorem{thrm}{Theorem}[section]
\newtheorem{dfn}[thrm]{Definition}
\newtheorem{prps}[thrm]{Proposition}
\newtheorem{lem}[thrm]{Lemma}
\newtheorem{rem}[thrm]{Remark}
\newtheorem{cor}[thrm]{Corollary}

\newcounter{fig}
\newenvironment{fig}[1][]{\refstepcounter{fig}\par\medskip\noindent\center{}}

\newtheorem{theorem}{Theorem}[section]
\theoremstyle{plain}

\newtheorem{definition}[theorem]{Definition}

\newtheorem{lemma}[theorem]{Lemma}

\newtheorem{remark}[theorem]{Remark}
\numberwithin{equation}{section}

\title{Existence and Instability of Traveling Pulses of Keller-Segel System with Nonlinear Chemical Gradients and Small Diffusions}

\author{
Chueh-Hsin Chang\footnotemark[1]\ \footnotemark[5]
, Yu-Shuo Chen\footnotemark[2]\ \footnotemark[6]
, John M. Hong\footnotemark[3]\ \footnotemark[7]
, and Bo-Chih Huang\footnotemark[4]\ \footnotemark[8]
}

\begin{document}
\maketitle

\renewcommand{\thefootnote}{\fnsymbol{footnote}}
\footnotetext[1]{\scriptsize Department of Applied Mathematics, Tunghai University, Taichung 40704, Taiwan. Email:\hspace{-0.1cm} {\tt changjuexin@thu.edu.tw}.}
\footnotetext[2]{\scriptsize Department of Mathematics, National Central University, Taoyuan 32001, Taiwan. Email:\hspace{-0.1cm} {\tt formosa1502@gmail.com}.}
\footnotetext[3]{\scriptsize Department of Mathematics, National Central University, Taoyuan 32001, Taiwan. Email:\hspace{-0.1cm} {\tt jhong@math.ncu.edu.tw}.}
\footnotetext[4]{\scriptsize Department of Mathematics, National Chung Cheng University, Chiayi 62102, Taiwan. Email:\hspace{-0.1cm} {\tt huangbz@ccu.edu.tw}.}
\footnotetext[5]{\scriptsize The 1st author is partially supported by Ministry of Science and Technology, Taiwan under grants MOST-106-2115-M-029-001-MY2.}
\footnotetext[6]{\scriptsize The 2nd author is partially supported by Ministry of Science and Technology, Taiwan under grants MOST-106-2811-M-032-008.}
\footnotetext[7]{\scriptsize The 3rd author is partially supported by Ministry of Science and Technology, Taiwan under grants MOST-106-2115-M-008-006-MY2.}
\footnotetext[8]{\scriptsize The 4th author is partially supported by Ministry of Science and Technology, Taiwan under grants MOST-104-2115-M-194-012-MY3.}
\renewcommand{\thefootnote}{\arabic{footnote}}



\medskip

\begin{abstract}
\noindent {In this paper, we consider a generalized model of $2\times 2$ Keller-Segel system with nonlinear chemical gradient and small cell diffusion. The existence of the traveling pulses for such equations is established by the methods of geometric singular perturbation (GSP in short) and trapping regions from dynamical systems theory. By the technique of GSP, we show that the necessary condition for the existence of traveling pulses is that their limiting profiles with vanishing diffusion can only consist of the slow flows on the critical manifold of the corresponding algebraic-differential system. We also consider the linear instability of these pulses by the spectral analysis of the linearized
operators. } \newline
\end{abstract}


\section{Introduction}

In biology, every creature in the world, their response toward internal and external signals plays an important role in survival. As stated in Keller and Segel's paper \cite{KS}, "Bands of motile \emph{Escherichia coli} have been observed to travel at constant speed when the bacteria are placed in one end of a capillary tube containing oxygen and an energy source. Such bands are a consequence of a chemotactic mechanism." By definition, the chemotaxis is the movement of an organism for instance single-cell,  multicellular or bacteria in response to a chemical stimulus,  which permits the bacteria to seek an optimal environment: the bacteria avoid low concentrations and move preferentially toward higher concentrations of some critical substrate.

Mathematical modeling of chemotaxis has developed into a large and
diverse discipline, whose aspects include its mechanistic basis, the
modeling of specific systems and the mathematical behavior of the
underlying equations. The Keller-Segel system of chemotaxis
\cite{KS} has provided a cornerstone for much of this work, its
success being a consequence of its intuitive simplicity, analytical
tractability and capacity to replicate key behavior of chemotactic
populations. The general form of the model in \cite{HPa1} is
\begin{equation}
\label{gKS}
  \left\{\begin{split}
    u_t &= \nabla\cdot\big(k_1(u,v)\nabla u-k_2(u,v)u\nabla v\big)+k_3(u,v),\\
    v_t &= D_v\D v+k_4(u,v)-k_5(u,v)v.
  \end{split}\right.
\end{equation}
Here $u$ denotes the cell density, $v$ is the concentration of the chemical signal. $k_1,k_2$ and $k_3$ are diffusivity of the cells, the chemotactic sensitivity, and the cell growth and death respectively. The positive constant $D_v$ denotes the diffusion coefficient, and $k_4,k_5$ describes the production and degradation of the chemical signal respectively. Note that cell migration is dependent on the gradient of the signal.

In this paper, we consider the following one-dimensional Keller-Segel system with nonlinear gradients of chemical signal concentration and small cell diffusion:
\begin{equation}
\label{KS}
\left\{
 \begin{aligned}
   &u_t = (\epsilon u_x - \chi u\phi(v_x) )_x,\\
   &\epsilon v_t = v_{xx} + u -g(v),
 \end{aligned}
 \right.
\end{equation}
where $\chi>0$, $\phi$ and $g$ are functions of their variables, and $0<\epsilon \ll 1$. The term $\chi u\phi(v_x)$ stands for the nonlinear interaction between cells and chemical signals
and the function $g(v)$ stands for the nonlinear degradation of the chemical signal. In the general model \eqref{gKS}, the function $k_5$ represented as the degradation rate. The simplest one is to choose $k_5$ be a constant and  it is appeared in most papers \cite{CPZ,HPa1,HPo,H1,HWa,N1,NSY,SS}. It is interesting to consider a more complex degradation rate such as $g(v)$ is a cubic polynomial. Particularly, we are interested in the situation when the degradation is non-monotone. For this reason, we consider $g(v)$ is quadratic first. In the further we will study the case for $g(v)$ is cubic.

The small parameter $\epsilon$ appearing in system \eqref{KS} describes the situation when the cell under small motility, the chemical signal achieves the steady state much faster than the cell.

Let us review some previous results related to
\eqref{KS} and clarify the motivation of the study
of this equation. The majority of this work has been
devoted to a special case in which the functions $k_i$ are assumed
to have linear form. We refer this model as the minimal model for
Keller-Segel system named by Childress and Percus
\cite{ChPe}. The model is as follows:
\begin{equation}
\label{min}
  \left\{\begin{split}
    u_t &= \nabla\cdot\big(D_1\nabla u-\chi u\nabla v\big),\\
    v_t &= D_2\D v+f(u)-g(v),
  \end{split}\right.
\end{equation}
for some constants $D_1,D_2,\chi>0$. The minimal model has rich and interesting properties including globally existing solutions, finite time blow-up and spatial pattern formation. If $D_1=D_2=1$, and $f(u)=au,\ g(v)=bv$ for some $a,b\geq0$, system \eqref{min} on the bounded domains $\Omega\subset\RR^N$ with zero flux boundary condition has been shown that the qualitative behavior of solutions strongly depends on the space dimension. In one dimensional case ($N=1$), all solutions exist globally, which is proved recently in \cite{HPo,OY}. For $N=2$, let us denote
$$
  m:=\int_{\Omega}u_0dx,
$$
where $u_0$ is the initial cell density. If $m<4\pi$, the solution will exist globally in time and bounded \cite{OY}. Otherwise for any $m>4\pi$ which satisfying $m\not\in\{4k\pi\big| k\in\NN\}$, there is an initial data $(u_0,v_0)^T$ with $\int_{\Omega}u_0dx=m$ such that the corresponding solution for \eqref{min} blows up either in finite time or infinite time when $\Omega$ is simply connected \cite{HWa,SS}. In \cite{CPZ}, the critical function space for system \eqref{min} is $L^{N/2}(\Omega)$. In other words, it means that a threshold value exists such that initial condition below threshold in the $L^{N/2}$-norm lead to global existing solutions, whereas initial data above threshold lead to finite blow up. For the parabolic-elliptic variants of \eqref{min}
\begin{equation}
\label{pemin}
  \left\{\begin{split}
    u_t &= \D u-\chi\nabla\cdot(u\nabla v),\\
    0 &= \D v+au-bv,
  \end{split}\right.
\end{equation}
with zero flux boundary condition, it is proven in \cite{JL,N1,N2} that some solutions of \eqref{pemin} in two-dimensional domain blow up in finite time if $m$ is large enough. Whereas if $m$ is small, then solutions remain bounded. The precise threshold value for $m$ could be $8\pi$ when considering the radially symmetric solution and $4\pi$ in the general case. In three-dimensional case, finite time blow up in system \eqref{pemin} may occur without any threshold \cite{HMV,NSY,NS}. More detail can be found in the survey of Horstmann \cite{H1} and the references cited therein.

In \eqref{min} the interaction term $\chi u\nabla v$ is linear in
$\nabla v$ which may allow unbounded advective
velocities. In fact it is an unrealistic depiction of individual
cell behavior since cells have a maximum velocity. Due to this
reason, some authors Patlak \cite{Patlak} in the 1950s have derived
models for chemotaxis based on a more realistic description of
individual cell migration (see also \cite{A,Othmer2}). Different
macroscopic descriptions for a population flux were derived
according to different types of migration (\emph{flagella bacteria}
or \emph{leukocytes}). Under certain assumptions, a number of
generalized Keller-Segel equations were derived in which the
advective velocity in \eqref{min} was replaced with a form depending
nonlinearly on the signal gradient, for example
\begin{equation}
\label{nongrad}
  \left\{\begin{split}
    u_t &= \nabla( D \nabla u -\chi u \mathbf{\Phi}_c(\nabla v)),\\
    v_t &= \D v + u - g(v).
  \end{split}\right.
\end{equation}
$$
 \mathbf{\Phi}_c = \frac{1}{c}\left( \tanh\left(\frac{cv_{x_1}}{1+c}\right),\cdots,\tanh\left(\frac{cv_{x_N}}{1+c}\right) \right)
$$
where $D,\chi,c>0$ for the case of the flagella bacteria \emph{Escherichia coli}. Of course, a variety of other choices for $\mathbf{\Phi}_c$ may also be appropriate.
Motivated by the nonlinear gradient model \eqref{nongrad}, we assume that the functions $\p$ and $g$ in \eqref{KS} satisfy the following conditions:
\begin{enumerate}
 \item[($A_1$)] $\p\in C^1(\RR)$ is a monotonic function with $\p'>0$.
 \item[($A_2$)] $g\in C^0(\RR_+)$ is a positive function and there exists exactly one $\b>0$ such that $g(v)$ are strictly decreasing if $v<\beta$ and $g(v)$  is strictly increasing if $v>\beta$.
\end{enumerate}

\begin{remark}
If $\phi(0) = 0$ then the system \eqref{KS} is similar to the minimal model. Thus from now on we assume that $\phi(0)>0$.
\end{remark}
Our goal in this paper is to study the existence and stability of the traveling wave solutions of \eqref{KS} which are in the following form
\begin{equation*}
  (U,V)^T=\big(U(\xi), V(\xi)\big)^T, \quad \xi
\equiv x-st,
\end{equation*}
where $s$ represents the traveling speed of $(u,v)$. We impose the initial conditions for \eqref{KS} as
\begin{equation*}
\label{inicond}
  (u,v)^T(x,0)=(u_0,v_0)^T(x)\to (u_{\pm},v_{\pm})^T\text{ as }x\to\pm\infty.
\end{equation*}
Since $u$ and $v$ in \eqref{KS} represent the biological quantities, we will restrict to the biologically relevant regime in which $u_{\pm},v_{\pm}\ge 0$.

The first result for the existence of one-dimensional traveling wave solutions of the Keller-Segel system can be tracked back to 1971. Most of results \cite{EFN,KO,KS} about traveling waves, especially the traveling pulses, are obtained when one considers the logarithmic sensitivity as
$$
  \left\{\begin{array}{l}
    u_t=\big(D u_x-\chi u\ln(v)_x\big)_x, \\
    v_t=d_vv_{xx}-k u,
  \end{array}\right.
$$
where $d_v$ is treated as a small parameter in \cite{EFN}. In \cite{HoS}, the author consider the model
\begin{equation*}
\label{gKS2}
  \left\{\begin{array}{l}
    u_t=\big(k(u)u_x-u\p(v)v_x\big)_x, \\
    v_t=d_vv_{xx}+g(u,v),
  \end{array}\right.
\end{equation*}
and gave suitable conditions of functions $k,\p,g$ that are needed for the existence of traveling wave and pulse patterns. In their results, the chemotactic sensitivity function $\p$ is constructed as $\p(v)=\frac{d}{dv}(\ln\Phi(v))$, where $\Phi(v)$ is a rational function. Contrast to the results in \cite{EFN,JLW,KO,KS,LiW}, we focus on the general nonlinear gradient sensitivity model with small parameter \eqref{KS}. By using the technique of geometric singular perturbation (GSP for short) and trapping region, we can construct the traveling pulse solution for system \eqref{KS}.

It seems difficult to obtain the traveling pulse solutions for the minimal model. For instance, in \cite{HoS} the authors consider the one-dimensional minimal model as follows
\begin{equation}
\label{1dmin}
  \left\{\begin{split}
    u_t &= (u_x-uv_x)_x,\\
    v_t &= v_{xx}+u-g(v),
  \end{split}\right.
\end{equation}
Plug the traveling wave ansatz $\xi=x-st,\ s>0$ into the first equation, it follows that
$$
  u(\xi)=e^{v(\xi)}e^{-s\xi}.
$$
Suppose system \eqref{1dmin} has traveling pulse solution such that $(U(\xi),V(\xi))^T\to (u_{\pm},v_{\pm})^T$ as $\xi\to\pm\infty$ with $0\le u_{\pm},v_{\pm}<\infty$ and $u_-=u_+$. Then as $\xi\to\infty$, $u_+=e^{v_+}\cdot 0=0=u_-$. But when $\xi\to-\infty$,  we obtain
$$
  \lim\limits_{\xi\to-\infty}e^{V(\xi)}e^{-s\xi}=u_-=0
$$
which implies that $V(\xi)\to-\infty$ as $\xi\to-\infty$ and leads to a contradiction. In the contrary, if we consider the minimal model with small parameter as follows:
\begin{equation}
\label{3dmin}
  \left\{\begin{split}
    u_t &=\nabla\cdot(\e\nabla u-\chi u\nabla v),\\
    \e v_t &= \D v+u-g(v),
  \end{split}\right.
\end{equation}
where the function $g$ satisfies the condition ($A_2$). It turns out that system \eqref{3dmin} can be reduced to a special case for our system \eqref{KS} (see Remark \ref{rem2.1}), and we can obtain the traveling pulse solutions for system \eqref{3dmin} by using GSP theory and trapping region method. To our knowledge, this is the first result for the existence of the traveling pulse solutions for the minimal model.


In this article we also study the stability of the above traveling pulse solutions. The stability here means the linear stability (see, for example, \cite{AGJ} and \cite{Volpert}%
). That is, we consider the small perturbation of the traveling wave
solutions with exponentially decay behaviors at $\pm \infty $. Then
the linear stability is determined by the spectral analysis of the
linearized operators around the solutions. The stability analysis of
traveling waves were well-studied in reaction-diffusion systems
\cite{Volpert}. However, for Keller-Segel type systems there are few
results. Nagai and Ikeda \cite{NagaiIkeda} obtained the linearly
unstable results. There are also other results about stability, such
as \cite{RosenBaloga} in which they studied a special perturbation,
Li et al. \cite{LiLiWang} obtained the nonlinear stability for other
kind Keller-Segel equations. Li and Wang also had other results
about nonlinear stability which can be seen in the reference of
\cite{LiLiWang}. In a first step, we obtain the essential spectrum from the behaviors of linearized operators at $\pm\infty$ in the function space $X:=BUC(\RR;\RR^2)$, where%
\begin{equation}
\label{X}
  BUC(\RR;\RR^2):=\{\mathbf{p}:\RR\to\RR^2\big| \mathbf{p}\text{ is bounded and uniformly continuous}\}.
\end{equation}
By considering the standard theory of essential spectrum as stated in Theorem
A.2 in Chapter 5 of \cite{Henry} which give us the boundaries of the essential
spectrum, we can find that the essential
spectrum have elements with positive real parts, therefore the
traveling pulse solution is unstable in the Lyapunov sense as the
usual way. To establish the more detailed stability result, we also
consider the stability of pulse solutions in the weighted
function spaces 
\begin{equation}
\label{Xw}
  X_w:=\{\mathbf{p}:\RR\to\RR^2\big| w\mathbf{p}\in X\}
\end{equation}
with $w=e^{\rho z}$ for some $\rho >0$ to be determined later. In
the cases about stability in $X_{w},$ the traveling wave solutions cannot be
stabilized by introducing the weight and therefore we
can also find elements of essential spectrum with positive real
parts. Hence the traveling pulse solutions is still unstable in the space
$X_{w}$.

This paper is organized as follows. In Section 2 we prove the existence of traveling wave solutions. In Section 3 we give the proof of the instability parts.

\section{Traveling Wave Solutions to Keller-Segel equations}

In this section, we study the traveling wave solutions of system \eqref{KS}. 

\subsection{A dynamical system formulation}
Let $\xi=x-st$, we look for the profile of the traveling waves solution of system \eqref{KS} of the form
$$
  (u,v)^T(x,t)=(U,V)^T(\xi),\quad \xi=x-st
$$
with $U,V,\in C^{\infty}(\RR)$ satisfying boundary conditions
\begin{equation}
\label{bdcond}
  U(\pm\infty)=u_{\pm},\ V(\pm\infty)=v_{\pm},\ U'(\pm\infty)=V'(\pm\infty)=0,
\end{equation}
where $Z':=\frac{dZ}{d\xi}$ and $s$ represents the traveling speed of $(U,V)$. Here we are only interested in the case $U,V\ge 0$ due to the biological relevance.

\begin{dfn}
\label{def2.1}
The traveling wave profile $U$ is said to be a pulse if $u_-=u_+$ and is a front if $u_-\neq u_+$.
\end{dfn}

Plugging $(U,V)^T(\xi)$ into system \eqref{KS}, we have the following system
\begin{equation}
\label{KSODE}
\left\{
 \begin{aligned}
   &-sU' = (\epsilon U' - \chi U\phi(V'))',\\
   &-\epsilon sV' = V'' + U -g(V).
 \end{aligned}
 \right.
\end{equation}
Integrating the first equations in system \eqref{KSODE} and using the boundary conditions \eqref{bdcond}, we have that
\begin{equation*}
  -s(U-u_-)=\e U'-\chi U\p(V')+\chi u_-\p(0).
\end{equation*}
According to above equality, we have the following system
\begin{equation}
\label{KSODE2}
  \left\{\begin{split}
    \e U'   &= -sU+\chi U\p(V')+su_--\chi u_-\p(0), \\
    -\e s V'&= V''+U-g(V),
  \end{split}\right.
\end{equation}
To solve the equation \eqref{KSODE2} with boundary condition \eqref{bdcond}, we let $W:=V'$, then, system \eqref{KSODE2} can be reformulated as singular perturbation system,
\begin{equation}
\label{slow}
  \left\{\begin{split}
    \e U' &=-sU+\chi U\p(W)+su_--\chi u_-\p(0), \\
    W'     &=-\e s W-U+g(V), \\
    V'     &= W,
  \end{split}\right.
\end{equation}
The system \eqref{slow} is called {\it slow system}. Let us define the fast variable $\zeta :=\frac{\xi}{\e}$, then the corresponding {\it fast system} is read as
\begin{equation}
\label{fast}
  \left\{\begin{split}
    \dot{U} &=-sU+\chi U\p(W)+su_--\chi u_-\p(0), \\
    \dot{W} &=\e(-\e s W-U+g(V)), \\
    \dot{V} &=\e W,
  \end{split}\right.
\end{equation}
where $\dot{Z}:=\frac{dZ}{d\zeta}$. According to the boundary condition \eqref{bdcond}, we impose the boundary conditions for \eqref{slow} (or equivalently \eqref{fast}) as
\begin{equation}
\label{bdcond2}
  \left\{\begin{array}{l}
    \big(s-\chi\p(0)\big)(u_--u_+)=0, \\
    u_+-g(v_+)=u_--g(v_-)=0.
  \end{array}\right.
\end{equation}

To establish the existence of the traveling wave solution of \eqref{KS} is equivalent to find an orbit $\L_{\e}$ of \eqref{slow} (or equivalently \eqref{fast}) in $\RR_+\times\RR\times\RR_+$ that connects $(u_+,0,v_+)^T$ to $(u_-,0,v_-)^T$.

\begin{dfn}\
\label{def2.2}
\begin{itemize}
\item[\rm{(1)}] An asymptotic state is a solution $(U,V,W)^T(\xi)$ of system \eqref{slow} {\rm(}or equivalently \eqref{fast}{\rm)} with $\e=0$, and satisfies the limit $(U,V,W)^T(\xi)\to(u_{\pm},0,v_{\pm})$ as $\xi\to\pm\infty$.
\item[\rm{(2)}] An asymptotic state $(U,V,W)^T(\xi)$ admits a {\it traveling wave profile} if there is a solution $(U,V,W)^T(\xi,\e)$ of system \eqref{KS} such that the orbit of $(U,V,W)^T(\xi,\e)\to(U,V,W)^T(\xi)$ as $\e\to0$ in $L^1(K)$ for any compact subset $K$ of $\RR$.
\end{itemize}
\end{dfn}

In terms of singular perturbation theory, the orbit $\L_0$ of the asymptotic state is the {\it singular orbit} of its traveling wave profile $\L_{\e}$.

We employ the geometric singular perturbation theory to analyze the system \eqref{slow} (or equivalently \eqref{fast}) and show the existence of the traveling wave solution of \eqref{KS}. The basic ideas of GSP are:
\begin{itemize}
\item[(1)] Examining the dynamics of the limiting systems for $\e=0$ in different scales (slow and fast scales);
\item[(2)] constructing singular orbits that consist of orbits of limiting slow and limiting fast systems - those are nothing but asymptotic states;
\item[(3)] showing that there exist true orbits shadowing the singular orbits - the existence of traveling wave profiles of asymptotic states.
\end{itemize}

\begin{remark}
\label{rem2.1}
For the multi-dimensional minimal model with small parameter $0<\e\ll 1$ as the system \eqref{3dmin}
\begin{equation*}
  \left\{\begin{split}
    u_t &=\nabla\cdot(\e\nabla u-\chi u\nabla v),\\
    \e v_t &= \D v+u-g(v).
  \end{split}\right.
\end{equation*}
Suppose $(u,v)^T(\tx,t)=(U,V)^T(\xi),\ \xi=\tx\cdot\mathbf{\Psi}-st$ {\rm(}$\tx\in\RR^N,N\ge 1,t\ge 0${\rm)} is a traveling wave solution of \eqref{3dmin} connecting $(u_+,v_+)^T$ and $(u_-,v_-)^T$ and propagating in the direction $\mathbf{\Psi}\in\mathbb{S}^{N-1}$.
It is not hard to see that $U,V$ satisfies the ODE system as
\begin{equation*}
  \left\{\begin{aligned}
     &-sU' = (\epsilon U' - \chi UV')',\\
     &-\epsilon sV' = V'' + U -g(V),
  \end{aligned}\right.
\end{equation*}
which is a special case of \eqref{KSODE} when $\p$ is the identity function. Therefore, by showing the existence of the traveling wave solution of the system \eqref{KS}, we can also show the existence of the traveling wave solution of the multi-dimensional minimal model \eqref{3dmin} with small parameter $0<\e\ll 1$.
\end{remark}

\subsection{Critical manifold, and limiting fast dynamics}

From the boundary condition \eqref{bdcond2}, we have the following
two possible cases:
\begin{itemize}
\item[(1)] If $u_-\neq u_+$, then the traveling speed $s=\chi\p(0)$, and we have the traveling front solution.
\item[(2)] If $s\neq\chi\p(0)$, then $u_-=u_+$ and thus we obtain the traveling pulse solution.
\end{itemize}
In this paper, we only focus on the traveling pulse case. Without loss of generality, we assume that $s>\chi\p(0)$.

Let us consider the slow and fast system \eqref{slow} and \eqref{fast} from geometric singular perturbation point of view. When $\e\to0$, we have the limiting slow and fast systems
\begin{equation}
\label{limslow}
  \left\{\begin{split}
    0 &=-sU+\chi U\p(W)+su_--\chi u_-\p(0), \\
    W'     &=-U+g(V), \\
    V'     &= W,
  \end{split}\right.
\end{equation}
and
\begin{equation}
\label{limfast}
  \left\{\begin{split}
    \dot{U} &=-sU+\chi U\p(W)+su_--\chi u_-\p(0), \\
    \dot{W} &=0, \\
    \dot{V} &=0,
  \end{split}\right.
\end{equation}
respectively. Note that the limiting slow system \eqref{limslow} is an algebraic-differential system consisting of one algebraic equation and two differential equations. The solutions of such algebraic equation form a two-dimensional manifold in the phase space, which is so-called the {\it critical manifold} defined as
\begin{equation*}
\label{critical}
  \mathcal{M}_0=\{(U,V,W)^T\in\RR_+\times\RR\times\RR_+\big|U(\chi\p(W)-s)=u_-(\chi\p(0)-s)\},
\end{equation*}
and the orbits of limiting slow system lie on $\mathcal{M}_0$. On the other hand, it is clear the points on critical manifold $\mathcal{M}_0$ are equilibria of the limiting fast system \eqref{limfast}. Moreover, the linearization of the system \eqref{limfast} along $\mathcal{M}_0$ is
$$
  \left[\begin{array}{ccc}
    \frac{u_-}{U}(\chi\p(0)-s) & \chi U\p'(W) & 0 \\
    0 & 0 & 0 \\
    0 & 0 & 0
  \end{array}\right]
$$
The eigenvalue $\l$ in the transversal direction of $\mathcal{M}_0$ is
\begin{equation}
\label{treigen}
  \l=\frac{u_-}{U}(\chi\p(0)-s)<0,
\end{equation}
by our assumption $s>\chi\p(0)$. The critical manifold $\mathcal{M}_0$ is {\it normally hyperbolic} and each point on $\mathcal{M}_0$ is a global attractor for the system \eqref{limfast}. Hence, if there is a singular orbit $\L_0$ of limiting slow system \eqref{limslow} connecting $(u_-,0,v_-)^T$ and $(u_+,0,v_+)^T=(u_-,0,v_+)^T$, then, according to the geometric singular perturbation theory, the asymptotic state $\L_0$ admits a traveling wave profile $\L_{\e}$, such that $\L_{\e}\to\L_0$ as $\e\to 0$. Moreover, \eqref{treigen} shows that all points on $\mathcal{M}_0$ are global attractors for the system \eqref{limfast}. Because both $(u_-,0,v_-)^T,(u_-,0,v_+)^T\in\mathcal{M}_0$, the existence of the traveling pulse solution of \eqref{KS} is equivalent to the existence of an singular orbit $\L_0$ of the system \eqref{limslow} connecting $(u_-,0,v_-)^T$ and $(u_+,0,v_+)^T=(u_-,0,v_+)^T$.

\subsection{The limiting slow dynamics}

From the algebraic equation in \eqref{limslow}, we can represent $U$ in terms of $W$ as
\begin{equation}
\label{utow}
  U=h(W):=\frac{u_-(\chi\p(0)-s)}{\chi\p(W)-s}.
\end{equation}

\begin{lem}
\label{lem2.3}
$h$ is an strictly increasing function.
\end{lem}
\begin{proof}
Since $s>\chi\p(0)$ and $\p'>0$, then
$$
  h'(W)=-\frac{u_-\chi\p'(W)(\chi\p(0)-s)}{(\chi\p(W)-s)^2}>0.
$$
\end{proof}

Plug \eqref{utow} into the limiting slow system, then we get the reduced system on the critical manifold $\mathcal{M}_0$ as
\begin{equation}
\label{redode}
  \left\{\begin{split}
    V'&= W, \\
    W'&=-h(W)+g(V).
  \end{split}\right.
\end{equation}
The solution for \eqref{limslow} can be obtained by solving the
system \eqref{redode} and using the equality \eqref{utow}. Hence, to
find the singular orbit of \eqref{limslow} connecting
$(u_+,0,v_-)^T$ to $(u_-,0,v_+)^T$ is equivalent to find an orbit of
\eqref{redode} connecting $(v_-,0)^T$ to $(v_+,0)^T$ for any fixed
$u_-$.

It is easy to see that the equilibria of \eqref{redode} are of the form $E=(V,0)^T$. We have the following proposition:

\begin{prps}
\label{prop2.4}
Suppose $E=(V,0)^T$ is an equilibrium of the system \eqref{redode}. Then $E$ is a saddle if $V>\b$ and $E$ is either a stable focus or a stable node if $V<\b$.
\end{prps}
\begin{proof}
Observe that the linearization of the system \eqref{redode} around $E=(V,0)^T$ is
$$
  A(E)=\left[\begin{array}{cc}
    0 & 1 \\
    g'(V) & -h'(0)
  \end{array}\right].
$$
The eigenvalues and the eigenvectors of $A(E)$ are
$$
  \begin{array}{ll}
    \l_1(E)=\frac{-h'(0)+\sqrt{h'(0)^2+4g'(V)}}{2}, & \tv_1(E)=(1,\l_1(E))^T, \\
    \l_2(E)=\frac{-h'(0)-\sqrt{h'(0)^2+4g'(V)}}{2}, & \tv_2(E)=(1,\l_2(E))^T.
  \end{array}
$$
According to Lemma \ref{lem2.3}, we know that $h'(0)>0$. Because $g'<0$ if $V<\b$ and $g'>0$ if $V>\b$, we can conclude that
$$
  \left\{\begin{array}{ll}
    \l_1(E)>0,\ \l_2(E)<0 & \text{if }V>\b, \\
    \l_{1,2}(E)<0\text{ or }\l_{1,2}(E)\in\CC\text{ with } \re\l_{1,2}(E)<0 &\text{if }V<\b.
  \end{array}\right.
$$
Therefore, $E$ is a saddle if $V>\b$ and $E$ is either a stable focus or a stable node if $V<\b$.
\end{proof}

\begin{rem}
\label{rem2.5}
For the case $V=\b$, suppose that $E=(\b,0)^T$ is an equilibrium of the system \eqref{redode} and suppose $g'(\b)=0$. The eigenvalues of $A(E)$ are
$$
  \l_1(E)=0,\quad \l_2(E)=-h'(0)
$$
This indicates that $E$ is a non-hyperbolic equilibrium for the system \eqref{redode}.
\end{rem}

To determine all equilibria for the system \eqref{redode}, let us
denote $g_{\infty}:=\lim\limits_{V\to\infty}g(V)$. By substituting
$W=0$ into the equation $W'=0$, we get $-h(0)+g(V)=-u_-+g(V)=0$. Hence the equilibrium of
\eqref{redode} occurs when $g(V)=u_-$. Therefore, according to
Proposition \ref{prop2.4}, we have three cases of equilibria as described in the following
proposition:

\begin{prps}\
\label{prop2.6}
\begin{itemize}
  \item[\rm{(1)}] If $g_{\infty}<g(0)$ {\rm(}see Figure 1.{\rm)}:
  \begin{itemize}
    \item[\rm{(i)}] if $u_-<g(\b)$ or $u_->g(0)$, then there is no equilibrium for \eqref{redode}.
    \item[\rm{(ii)}] if $u_-=g(\b)$ or $g_{\infty}<u_-<g(0)$, then there exists only one equilibrium $E=(v,0)^T$ for \eqref{redode} with $v\le\b$, and $E$ is either a stable focus, stable node, or nonhyperbolic equilibrium.
    \item[\rm{(iii)}] if $g(\b)<u_-<g_{\infty}$, then there are two equilibria $E_{\pm}=(v_{\pm},0)^T$ for \eqref{redode} with $0<v_+<\b<v_-,\ u_-=g(v_{\pm})$, and $E_-$ is a saddle, $E_+$ is either a stable focus or a stable node.
  \end{itemize}
  \item[\rm{(2)}] If $g_{\infty}>g(0)$ or $g_{\infty}=+\infty$ {\rm(}see Figure 1.{\rm)}:
    \begin{itemize}
    \item[\rm{(i)}] if $u_-<g(\b)$ or $u_-\ge g_{\infty}$, then there is no equilibrium for \eqref{redode}.
    \item[\rm{(ii)}] if $u_-=g(\b)$ or $g(0)<u_-<g_{\infty}$, then there exists only one equilibrium $E=(v,0)^T$ for \eqref{redode} with $v\ge\b$, and $E$ is either a saddle, or a nonhyperbolic equilibrium.
    \item[\rm{(iii)}] if $g(\b)<u_-\le g(0)$, then there are two equilibria $E_{\pm}=(v_{\pm},0)^T$ for \eqref{redode} with $0\le v_+<\b<v_-,\ u_-=g(v_{\pm})$, and $E_-$ is a saddle, $E_+$ is either a stable focus or a stable node.
  \end{itemize}
  \item[\rm{(3)}] If $g_{\infty}=g(0)$ {\rm(}see Figure 1.{\rm)}:
    \begin{itemize}
    \item[\rm{(i)}] if $u_-<g(\b)$ or $u_->g(0)$, then there is no equilibrium for \eqref{redode}.
    \item[\rm{(ii)}] if $u_-=g(\b)$, then there exist only one equilibrium $E=(v,0)^T$ for \eqref{redode} with $v=\b,\ u_-=g(\b)$, and $E$ is a nonhyperbolic equilibrium.
    \item[\rm{(iii)}] if $g(\b)<u_-<g(0)$, then there are two equilibria $E_{\pm}=(v_{\pm},0)^T$ for \eqref{redode} with $0<v_+<\b<v_-,\ u_-=g(v_{\pm})$, and $E_-$ is a saddle, $E_+$ is either a stable focus or a stable node.
  \end{itemize}
\end{itemize}
\end{prps}
\begin{fig}
\label{fig1}
\includegraphics[scale=0.3]{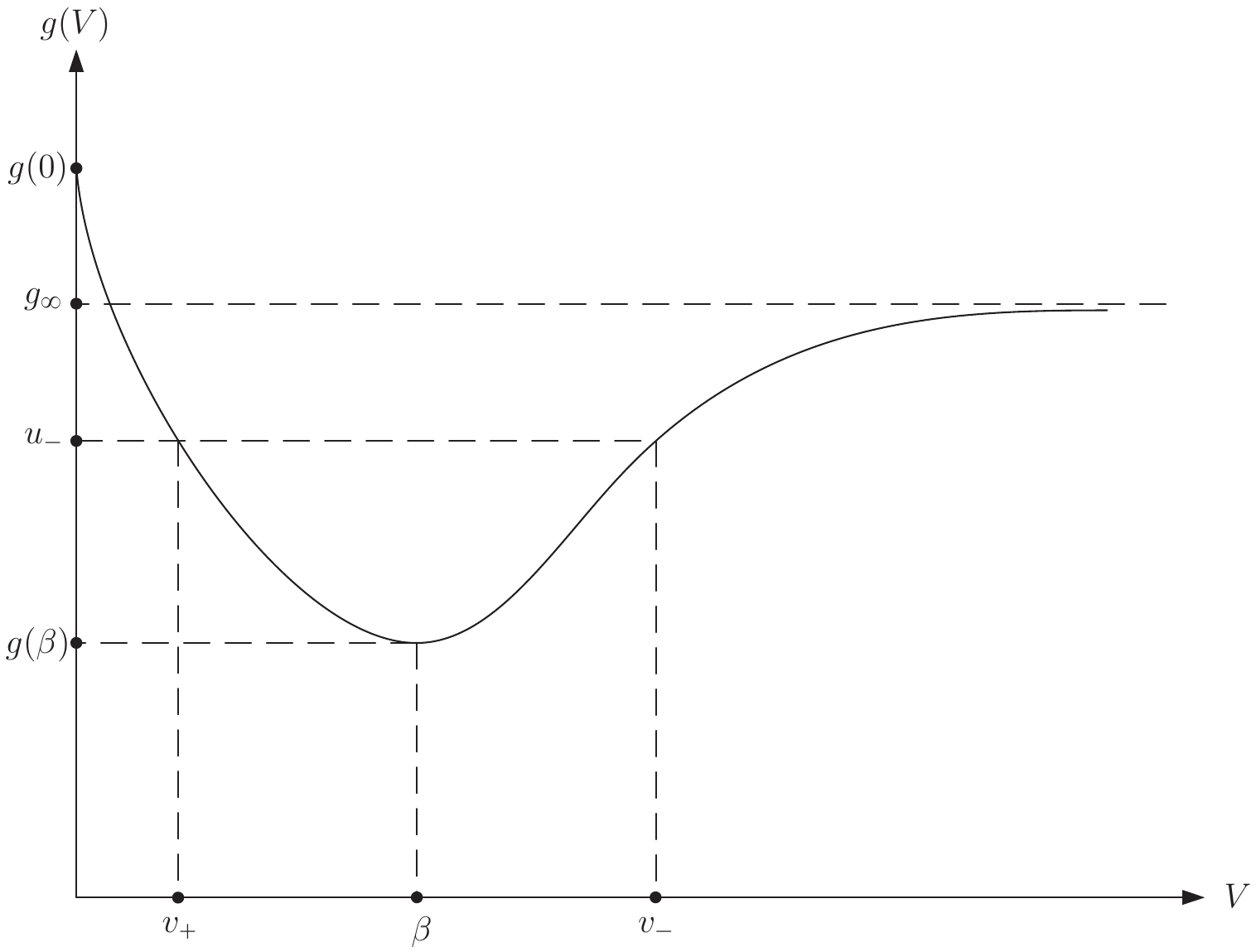}
\includegraphics[scale=0.3]{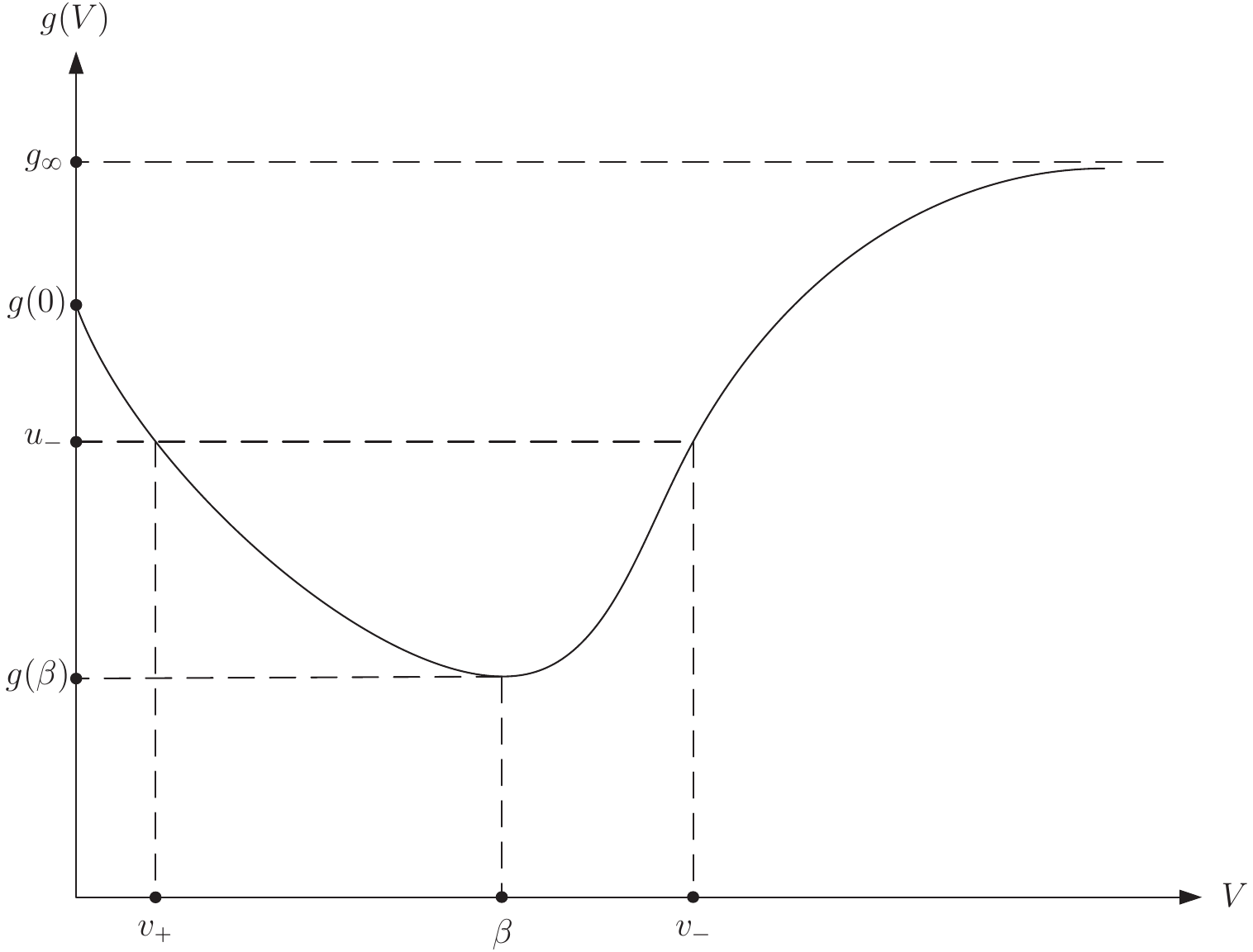}
\includegraphics[scale=0.3]{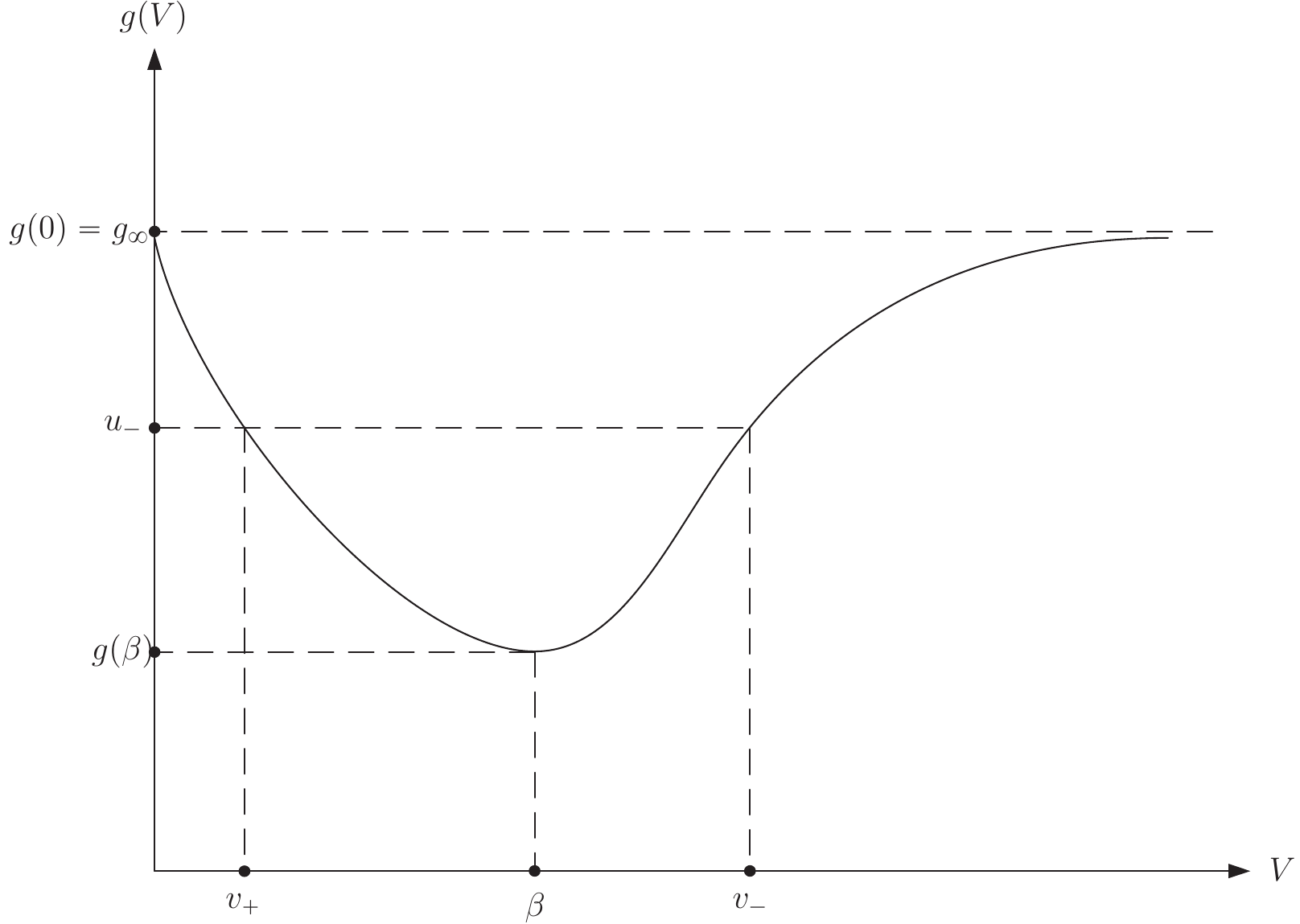}\\
Figure 1. The graph of $g$ from case (1) to case (3) respectively.\medskip
\end{fig}
\medskip

Let us denote
$$
  F(V,W):=(W,-h(W)+g(V))^T.
$$

\begin{prps}
\label{prop2.7} If there is only one equilibrium for the system \eqref{redode}, then \eqref{redode} has no closed orbit lying entirely in $\RR_+\times\RR$.
\end{prps}
\begin{proof}
According to Lemma \ref{lem2.3},
$$
  \nabla\cdot F=-h'(W)=\frac{\chi u_-(\chi\p(0)-s)}{(\chi\p(W)-s)^2}<0
$$
By Bendixson's criteria, system \eqref{redode} has no closed orbit lying entirely in $\RR_+\times\RR$.
\end{proof}

According to Proposition \ref{prop2.7}, there is no closed trajectory from $E$ to $E$ itself when $E$ is the only equilibrium of \eqref{redode}, which indicates that the case of traveling pulse solution with $v_-=v_+$ never occurs. According to Proposition \ref{prop2.6}, if there are two equilibria $E_{\pm}=(v_{\pm},0)^T$ for the system \eqref{redode}, then $v_+<\b<v_-$ and $E_+$ is a saddle, $E_-$ is a stable node or a stable focus. To find the orbit for the system \eqref{redode} connecting $E_-$ to $E_+$, we need to analyze the dynamics of the system \eqref{redode}. Hereafter, let us denote
$$
  g^*:=\min\{g(0),g_{\infty}\}.
$$

It is clear that the isocline for $V'=0$ is $W=0$. By setting
$W'=0$, we have $h(W)=g(V)$, that is,
$\frac{u_-(\chi\p(0)-s)}{\chi\p(W)-s}=g(V)$, it follows that $\p(W)=\frac{s}{\chi}+\frac{u_-}{g(V)}\big(\p(0)-\frac{s}{\chi}\big)$.
Since $\p$ is strictly monotonic with $\p'>0$, $\p$ has an inverse
function, and we can represent $W$ in terms of $V$ as
\begin{equation*}
\label{wtov}
  W=\p^{-1}(B(V)),\ \text{where}\ B(V):=\frac{s}{\chi}+\frac{u_-}{g(V)}\Big(\p(0)-\frac{s}{\chi}\Big).
\end{equation*}

\begin{prps}
\label{prop2.8}
For any fixed $u_-$ with $g(\b)<u_-\le g^*,\ u_-=g(v_{\pm})$, the isocline $W'=0$ is strictly decreasing if $V<\b$ and strictly increasing if $V>\b$ and $\p^{-1}(B(v_{\pm}))=0$. Moreover, $W'>0$ if and only if $W<\p^{-1}(B(V))$.
\end{prps}
\begin{proof}
It is clear that
$$
  B(v_{\pm})=\frac{s}{\chi}+\frac{u_-}{g(v_{\pm})}\Big(\p(0)-\frac{s}{\chi}\Big)=\p(0),
$$
so we have that $\p^{-1}(B(v_{\pm}))=0$.
For any $V\in\RR_+$, we have that
$$
  h(\p^{-1}(B(V)))=\frac{u_-(\chi\p(0)-s)}{\chi\p(\p^{-1}(B(V)))-s}=\frac{u_-(\chi\p(0)-s)}{\chi B(V)-s}=\frac{u_-(\chi\p(0)-s)}{s+\frac{u_-}{g(V)}\big(\chi\p(0)-s\big)-s}=g(V).
$$
Thus,
$$
  g'(V)=\frac{h'(\p^{-1}(B(V)))\cdot B'(V)}{\p'(\p^{-1}(B(V)))}.
$$
According to our assumption of $\p$ and Lemma \ref{lem2.3}, we have that $\p'>0,\ h'>0$, and thus $\sgn(B'(V))=\sgn(g'(V))$, that is, $B'(V)<0$ if $V<\b$ and $B'(V)>0$ is $V>\b$. Furthermore, if $W<\p^{-1}(B(V))$, since
$h\circ\p^{-1}\circ B=g$ and $h'>0$, then
$$
  W'=-h(W)+g(V)>-h(\p^{-1}(B(V)))+g(V)=-g(V)+g(V)=0.
$$
Consequently, we also have $W'<0$ if and only if $W>\p^{-1}(B(V))$.
\end{proof}

It is clear that $V'>0$ if $W>0$ and $V'\leq0$ if $W<0$. According to Proposition \ref{prop2.8}, the vector field for the system \eqref{redode} is shown in Figure \ref{fig2}.
\begin{fig}
\label{fig2}
\includegraphics[scale=0.5]{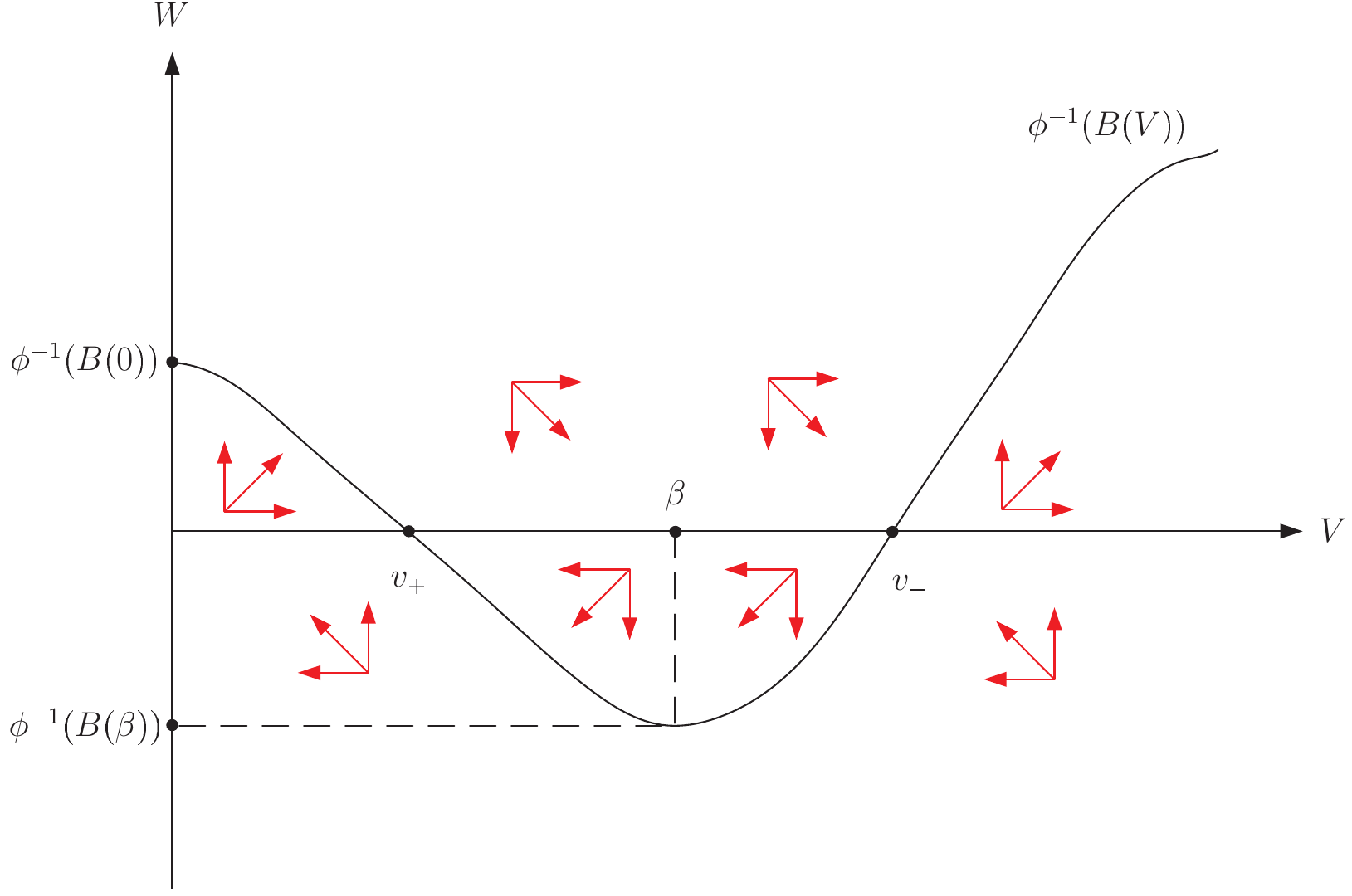}\\
Figure 2. Projection of the vector fields for the limiting slow
system \eqref{limslow} on VW-plane.\medskip
\end{fig}

\subsection{The trapping region}

In this subsection, for any fixed $u_-$ with $g(\b)<u_-\le g^*,\ u_-=g(v_{\pm})$, which depending on the case in Proposition \ref{prop2.6} (i.e. $g^*=g(0)$ or $g^*=g_{\infty}$), we find the orbit of \eqref{redode} connecting $E_-=(v_-,0)^T$ and $E_+=(v_+,0)^T$ by constructing a suitable trapping region, and using the Poincar\'{e}-Bendixson Theorem, and therefore prove the existence of the traveling pulse solution for system \eqref{KS} by geometric singular perturbation theory.

We first consider the case of $u_-$ satisfying $g(\b)<u_-<g^*,\ u_-=g(v_{\pm})$. In this case, we have $0<v_+<\b<v_-$. Moreover, by Proposition \ref{prop2.8}, we obtain
$$
  \p^{-1}(B(\b))<0,\qquad\p^{-1}(B(0))>0.
$$
Let us define the functions
\begin{equation*}
  \begin{split}
    J(V)&:=V(g(V)-u_-)\text{ on }[0,v_+], \\
    Q(V)&:=(v_--V)^2\Big(\inf_{z\in(V,v_-]}g'(z)\Big)\text{ on }(\b,v_-).
  \end{split}
\end{equation*}

\begin{lem}\
\label{lem2.9}
\begin{itemize}
  \item[{\rm(1)}]  If $\frac{1}{v_+}\int_0^{v_+}J(V)dV>\p^{-1}(B(\b))^2$, then there exists $v_*\in(0,v_+)$ such that $J(v_*)>\p^{-1}(B(\b))^2$.
  \item[{\rm(2)}]  If $\frac{1}{v_--\b}\int_{\b}^{v_-}Q(V)dV>\p^{-1}(B(0))^2$, then there exists $v^*\in(\b,v_-)$ such that $Q(v^*)>\p^{-1}(B(0))^2$.
\end{itemize}
\end{lem}
\begin{proof}
(1) Suppose $J(V)\le\p^{-1}(B(\b))^2$ for all $V\in[0,v_+]$. Then
$$
  \int_0^{v_+}J(V)dV\le\int_0^{v_+}\p^{-1}(B(\b))^2dV=v_+\p^{-1}(B(V))^2.
$$
The proof is complete. The result of (2) can be proved in a similar argument.
\end{proof}

\begin{prps}
\label{prop2.10}
Let us denote
\begin{eqnarray*}
    &&s_1:=\chi\Big(1-\frac{u_-}{g(\b)}\Big)^{-1}\Big[\p\Big(-\sqrt{\frac{1}{v_+}\int_0^{v_+}J(V)dV}\,\Big)-\frac{u_-\p(0)}{g(\b)}\Big], \label{s1} \\
    &&s_2:=\chi\Big(1-\frac{u_-}{g(0)}\Big)^{-1}\Big[\p\Big(\sqrt{\frac{1}{v_--\b}\int_{\b}^{v_-}Q(V)dV}\,\Big)-\frac{u_-\p(0)}{g(0)}\Big]. \label{s2}
\end{eqnarray*}
If $\chi\p(0)<s<s_*$ with $s_*:=\min\{s_1,s_2\}$, then there exist $v_*\in(0,v_+),\ v^*\in(\b,v_-)$ such that $-\sqrt{J(v_*)}<\p^{-1}(B(\b))<0$ and $0<\p^{-1}(B(0))<\sqrt{Q(v^*)}$.
\end{prps}
\begin{proof}
Since $g(\b)<u_-<g(0)$, then $1-\frac{u_-}{g(\b)}<0<1-\frac{u_-}{g(0)}$. If $s<s_*$, then $s<s_1,\ s<s_2$ and thus.
\begin{equation*}
  \begin{split}
    B(\b)&=\frac{s}{\chi}+\frac{u_-}{g(\b)}\Big(\p(0)-\frac{s}{\chi}\Big)>\p\Big(-\sqrt{\frac{1}{v_+}\int_0^{v_+}J(V)dV}\,\Big), \\
    B(0)&=\frac{s}{\chi}+\frac{u_-}{g(0)}\Big(\p(0)-\frac{s}{\chi}\Big)<\p\Big(\sqrt{\frac{1}{v_--\b}\int_{\b}^{v_-}Q(V)dV}\,\Big),
  \end{split}
\end{equation*}
which implies that $-\sqrt{\frac{1}{v_+}\int_0^{v_+}J(V)dV}<\p^{-1}(B(\b))<0$ and \\
$0<\p^{-1}(B(0))<\sqrt{\frac{1}{v_--\b}\int_{\b}^{v_-}Q(V)dV}$. Hence
$$
  \frac{1}{v_+}\int_0^{v_+}J(V)dV>\p^{-1}(B(\b))^2\quad\text{and}\quad\frac{1}{v_--\b}\int_{\b}^{v_-}Q(V)dV>\p^{-1}(B(0))^2.
$$
By Lemma \ref{lem2.9}, $v_*\in(0,v_+),\ v^*\in(\b,v_-)$ such that $J(v_*)>\p^{-1}(B(\b))^2,\ Q(v^*)>\p^{-1}(B(0))^2$. Therefore, $-\sqrt{J(v_*)}<\p^{-1}(B(\b))<0$ and $0<\p^{-1}(B(0))<\sqrt{Q(v^*)}$.
\end{proof}

\begin{prps}
\label{prop2.11}
If $\chi\p(0)<s<s_2$ and $v^*\in(\b,v_-)$ as in Proposition \ref{prop2.10}. If we choose $w^*$ such that $\p^{-1}(B(0))<w^*<\sqrt{Q(v^*)}$, then for equilibrium $E_-=(v_-,0)^T$ of the system \eqref{redode}, we have
$$
  \l_2(E_-)<-\frac{w^*}{v_--v^*}<0.
$$
\end{prps}
\begin{proof}
Recall that the equilibrium $E_-=(v_-,0)^T$ is a saddle for \eqref{redode}. The negative eigenvalues of $A(E_-)$ is $\l_2(E_-)=\frac{-h'(0)-\sqrt{h'(0)^2+4g'(v_-)}}{2}$. By the choice of $w^*$, we have that
$$
  {w^*}^2<Q(v^*)=(v_--v^*)^2\Big(\inf_{z\in(v^*,v_-]}g'(z)\Big).
$$
Since $g'(v_-)>0$ and $h'>0$, we have that
\begin{align*}
  \l_2(E_-)^2&=\frac{1}{2}h'(0)^2+g'(v_-)+\frac{1}{2}h'(0)\sqrt{h'(0)^2+4g'(v_-)} \\
             &>h'(0)+g'(v_-)>g'(v_-)\ge\inf_{V\in(v^*,v_-]}g'(V).
\end{align*}
That is, $\l_2(E_-)^2>\inf\limits_{V\in(v^*,v_-]}g'(V)>\frac{{w^*}^2}{(v_--v^*)^2}$, and hence $\l_2(E_-)<-\frac{w^*}{v_--v^*}<0$.
\end{proof}

We have the following theorem:

\begin{thrm}
\label{thm2.12} Assume that conditions
${\rm(}A_1{\rm)}\sim{\rm(}A_2{\rm)}$ hold. Let $u_-$ satisfy
$g(\b)<u_-<g^*$, where $g^*=\min\{g(0),g_{\infty}\}$ and
$u_-=g(v_{\pm})$. If the traveling speed $s$ satisfies
$\chi\p(0)<s<s_*$, then there exists a orbit $\L_0$ of system
\eqref{redode} from $E_-=(v_-,0)^T$ to $E_+=(v_+,0)^T$. The singular
orbit $\L_0$ admits a traveling wave profile $\L_{\e}$ for system \eqref{KS}.
\end{thrm}
\begin{proof}
We prove the theorem by constructing a suitable trapping region $\Omega$ containing $E_{\pm}$, and is enclosed by curves $\{\G_i\}_{i=1}^6$. According to Proposition \ref{prop2.10}, we can choose $v_*,v^*,w_*,w^*$ such that $v_*\in (0,v_+),\ v^*\in(\b,v_-)$, and $-\sqrt{J(v_*)}<w_*<\p^{-1}(B(\b))<0<\p^{-1}(B(0))<w^*<\sqrt{Q(v^*)}$. Let $\tn_i$ be the outward normal vector of $\G_i,\ i=1,\cdots,6$, where curves $\{\G_i\}$ are defined as follows
\begin{itemize}
  \item[(1)] $\G_1:=\big\{(V,W)\big| W=\frac{w_*}{v_*}V,\ 0\le V\le v_*\big\}$, with $\tn_1=\big(\frac{w_*}{v_*},-1\big)^T$.
  \item[(2)] $\G_2:=\{(V,W)\big| W=w_*,\ v_*\le V\le v_-\}$ with $\tn_2=(0,-1)^T$.
  \item[(3)] $\G_3:=\{(V,W)\big| V=v_-,\ w_*\le W<0\}$ with $\tn_3=(1,0)^T$.
  \item[(4)] $\G_4:=\big\{(V,W)\big| W=-\frac{w^*}{v_--v^*}(V-v_-),\ v^*\le V<v_-\big\}$, with $\tn_4=\big(\frac{w^*}{v_--v^*},1\big)^T$.
  \item[(5)] $\G_5:=\{(V,W)\big| W=w^*,\ 0\le V\le\b\}$ with $\tn_5=(0,1)^T$.
  \item[(6)] $\G_6:=\{(V,W)\big| V=0,\ 0<W<w^*\}$ with $\tn_6=(-1,0)^T$.
\end{itemize}
\begin{fig}
\label{fig3}
\includegraphics[scale=0.56]{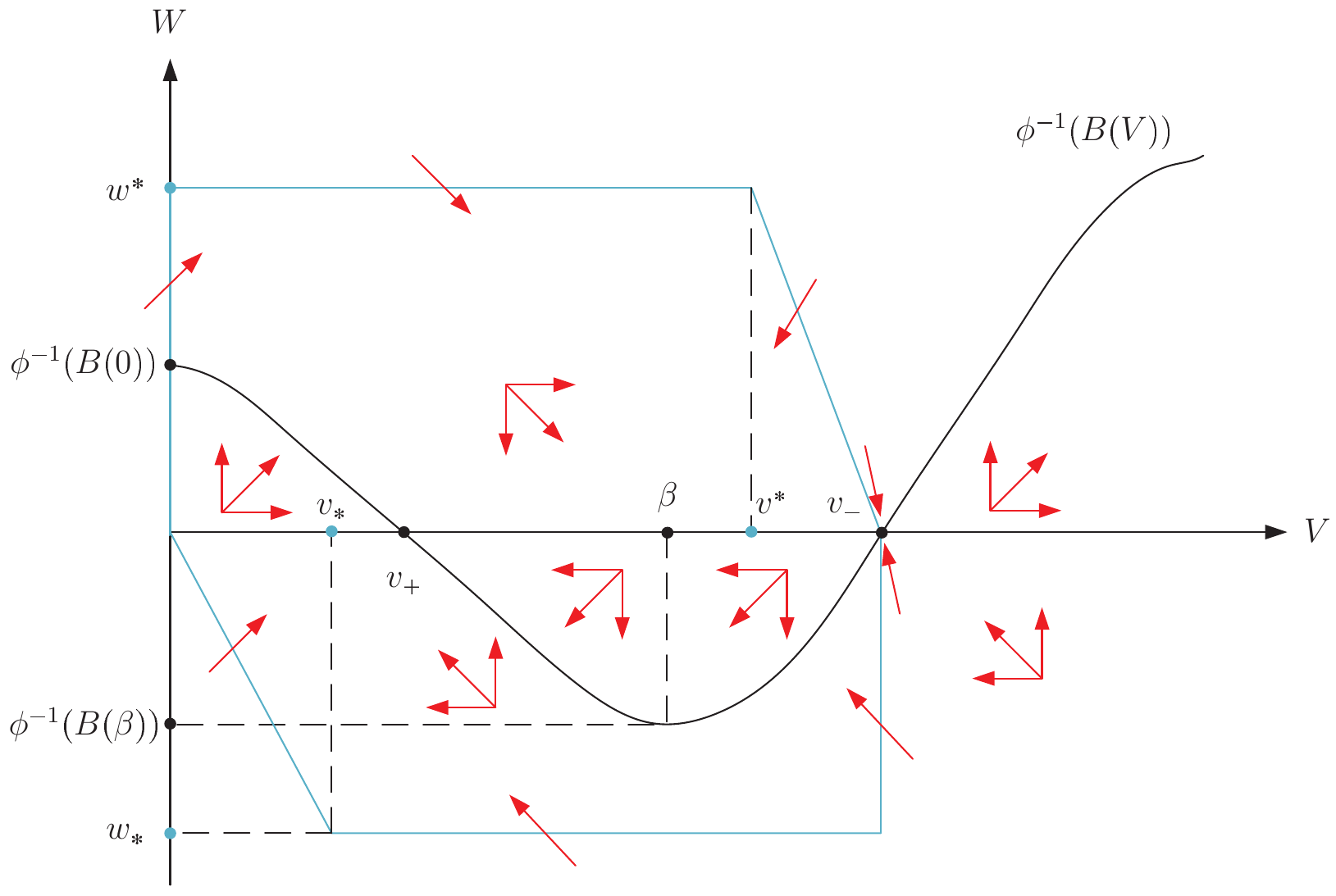}\\
Figure 3. The trapping region $\Omega$ for the system \eqref{redode}.\medskip
\end{fig}
To show the existence of the orbit, it is sufficient to show that the closed region $\Omega$ is a positively invariant region, which means that $\tn_i\cdot F\big|_{\G_i}\le 0$ for $i=1,\cdots,6$.
\begin{itemize}
\item[(1)] On $\G_1$: Since $h'>0$ and $W<0$ on $\G_1$, so $h(W)<h(0)=u_-$. Also, observe that
$$
  \frac{d}{dV}\Big(\frac{w_*^2}{v_*^2}V-g(V)\Big)=\frac{w_*^2}{v_*^2}-g'(V)>0.
$$
Then by choice of $w_*$, we have that
\begin{equation*}
  \begin{split}
    \tn_1\cdot F\big|_{\G_1} &=\frac{w_*^2}{v_*^2}V-g(V)+h(W) \\
                             &<\frac{w_*^2}{v_*^2}V-g(V)+u_-<\frac{w_*^2}{v_*}-g(v_*)+u_- \\
                             &<\frac{J(v_*)}{v_*}-g(v_*)+u_-=g(v_*)-u_--g(v_*)+u_-=0.
  \end{split}
\end{equation*}
\item[(2)] On $\G_2$: Since $h'>,\ w_*<\p^{-1}(B(\b))$, and by assumption ($A_2$) for $g$, we have that
$$
  \tn_2\cdot F\big|_{\G_2}=h(w_*)-g(V)<h(\p^{-1}(B(\b)))-g(V)=g(\b)-g(V)\le 0.
$$
\item[(3)] On $\G_3$: Since $w_*\le W<0$, it is clear that
$$
  \tn_3\cdot F\big|_{\G_3}=W<0.
$$
\item[(4)] On $\G_4$: We use Proposition \ref{prop2.11} to show that $\G_4$ excludes the stable subspace of $E_-$ from the region $\Omega$ so that all flows will go into the region $\Omega$. Since $h'>0$ and $W>0$ on $\G_4$, we have $h(W)>h(0)=u_-$. Also, by the choice of $w^*$, it gives
$$
  \frac{d}{dV}\Big(-\frac{{w^*}^2}{(v_--v^*)^2}(V-v_-)+g(V)\Big)=-\frac{{w^*}^2}{(v_--v^*)^2}+g'(V)>g'(V)-\Big(\inf_{V\in(v^*,v_-]}g'(V)\Big)\ge 0.
$$
Hence
\begin{equation*}
  \begin{split}
    \tn_4\cdot F\big|_{\G_4} &=-\frac{{w^*}^2}{(v_--v^*)^2}(V-v_-)+g(V)-h(W) \\
                             &<-\frac{{w^*}^2}{(v_--v^*)^2}(V-v_-)+g(V)-u_- \\
                             &\le -\frac{{w^*}^2}{(v_--v^*)^2}(v_--v_-)+g(v_-)-u_-=0.
  \end{split}
\end{equation*}
\item[(5)] On $\G_5$: Since $h'>0,\ w^*>\p^{-1}(B(0))$, and by assumption ($A_2$) of $g$, we have that
$$
  \tn_5\cdot F\big|_{\G_5}=-h(w^*)+g(V)<-h(\p^{-1}(B(0)))+g(V)=-g(0)+g(V)\le 0.
$$
\item[(6)] On $\G_6$: Since $0<W<w^*$, it is clear that
$$
  \tn_6\cdot F\big|_{\G_6}=-W<0.
$$
\end{itemize}
According to Pioncar\'{e}-Bendixson Theorem, for any fixed $u_-$
with $g(\b)<u_-<g^*,\ u_-=g(v_{\pm})$, there exists an orbit $\L_0$
of the system \eqref{redode} from $E_-=(v_-,0)^T$ to
$E_+=(v_+,0)^T$. The geometric singular perturbation theory implies
that, for $\e>0$ sufficiently small, there exists an invariant region $\Omega_{\e}$ close
to $\Omega$ such that the flow on $\Omega_{\e}$ is a regular
perturbation of the one on $\Omega$. From which we can conclude that,
there exists a heteroclinic orbit $\L_{\e}$ of the system
\eqref{slow} in a neighborhood of $\L_0$, in particular,
$\L_{\e}\to\L_0$ as $\e\to 0$. Such an orbit $\L_{\e}$ is the
traveling wave profile of $\L_0$, and hence is the traveling pulse
solution for the system \eqref{KS}.
\end{proof}

The existence of the traveling pulse for the symmetry case is given as follows.

\begin{cor}
\label{cor2.13} Suppose that conditions
${\rm(}A_1{\rm)}\sim{\rm(}A_2{\rm)}$ hold. Also assume that $u_-$
satisfies $g(\b)<u_-<g^*$ where $g^*=\min\{g(0),g_{\infty}\}$ and
$u_-=g(v_{\pm})$. If the traveling speed $s$ satisfies
$s^*<s<\chi\p(0)$ where $s^*:=\max\{s_1,s_2\}$ and
\begin{eqnarray*}
    &&s_1:=\chi\Big(1-\frac{u_-}{g(\b)}\Big)^{-1}\Big[\p\Big(\sqrt{\frac{1}{v_+}\int_0^{v_+}J(V)dV}\,\Big)-\frac{u_-\p(0)}{g(\b)}\Big], \\
    &&s_2:=\chi\Big(1-\frac{u_-}{g(0)}\Big)^{-1}\Big[\p\Big(-\sqrt{\frac{1}{v_--\b}\int_{\b}^{v_-}Q(V)dV}\,\Big)-\frac{u_-\p(0)}{g(0)}\Big],
\end{eqnarray*}
then there exists a orbit $\L_0$ of system \eqref{redode} connecting from $E_-=(v_-,0)^T$ to $E_+=(v_+,0)^T$. The singular orbit $\L_0$ admits a traveling wave profile $\L_{\e}$ of the system \eqref{KS}.
\end{cor}

\begin{rem}
\label{rem2.14}
If the condition ${\rm(}A_1{\rm)}$ is replaced by $\p'<0$, then we have similar result for the existence of the traveling pulse solution of the system \eqref{KS} as in Theorem \ref{thm2.12} and Corollary \ref{cor2.13}.
\end{rem}



\section{Instability of the traveling pulse solution}

In this section we consider the instability of the traveling pulse solution obtained in the previous section. We first consider the stability of $(U_{\e},V_{\e})^T(\xi)$ in the space $X=BUC(\RR;\RR^2)$ as mentioned in \eqref{X}. For convenience of the notation, we write $(U_{\e},V_{\e})^T$ as $(U,V)^T$ in the following.

Let $(U,V)^T(\xi)$ with $\xi =x-st$ be the traveling pulses of \eqref{KS}. Then $\left( U,V\right)^T$ satisfies
\begin{equation*}
\label{TWeq'n}
  \left\{\begin{array}{l}
    \epsilon U^{\prime \prime }+sU^{\prime }-\chi \left( U\phi \left( V^{\prime}\right) \right) ^{\prime }=0,\smallskip \\
    V^{\prime \prime }+s\epsilon V^{\prime }+U-g\left( V\right) =0.%
  \end{array}\right.
\end{equation*}
Suppose $(u,v)^T(x,t)$ is a perturbation of the traveling pulse solution in moving coordinate $\xi=x-st$:
$$
 (u,v)^T(x,t)=(U(\xi)+p(\xi,t),V(\xi)+q(\xi,t))^T.
$$
The linearized equation of \eqref{KS} around $(U,V)^T$ is
\begin{eqnarray*}
    &&\left\{\begin{array}{l}
        p_t=\e p_{\xi\xi}-\chi U\p'(V')q_{\xi\xi}+(s-\chi\p(V'))p_{\xi}-\chi(U\p'(V'))'q_{\xi}-\chi(\p(V'))'p, \\
        \e q_t=q_{\xi\xi}+s\e q_{\xi}+p-g'(V)q
       \end{array}\right. \\
    && \\
    &\Rightarrow&\left\{\begin{array}{l}
        p_t=\e p_{\xi\xi}-\chi U\p'(V')q_{\xi\xi}+(s-\chi\p(V'))p_{\xi}-\chi(U\p'(V'))'q_{\xi}-\chi(\p(V'))'p, \\
        q_t=\frac{1}{\e}q_{\xi\xi}+sq_{\xi}+\frac{1}{\e}p-\frac{1}{\e}g'(V)q
       \end{array}\right.
\end{eqnarray*}
The linearized operator $\mathcal{L}:\dom(\mathcal{L})\to X$ of (\ref{KS}) around $\left( U,V\right) $ is%
\begin{equation}
\label{L}
  \mathcal{L}:=D(\xi)\partial_{\xi\xi}+M(\xi)\partial_{\xi}+N(\xi)I,
\end{equation}
where
\begin{eqnarray*}
  D(\xi)&=&\left[\begin{array}{cc}
    \e & -\chi U(\xi)\p'(V'(\xi)) \\
    0  & \frac{1}{\e}
  \end{array}\right], \\
  M(\xi)&=&\left[\begin{array}{cc}
    s-\chi\p(V'(\xi)) & -\chi(U(\xi)\p'(V'(\xi)))' \\
    0                 & s
  \end{array}\right], \\
  N(\xi)&=&\left[\begin{array}{cc}
    -\chi(\p(V'(\xi)))' & 0 \\
    \frac{1}{\e}             & -\frac{1}{\e}g'(V(\xi))
  \end{array}\right],
\end{eqnarray*}
and $I$ denote the identity operator on $X$. Then the linearized equation of \eqref{KS} around $(U,V)^T$ can be written formally as a differential equation in $X$ as
$$
  \frac{d\tp}{dt}=\mathcal{L}\tp,\quad \tp=(p,q)^T\in X.
$$
We consider the spectral problem of the linearized operator $\mathcal{L}:$
$$
  \mathcal{L}\tp=\l\tp,\quad \tp=(p,q)^T\in X
$$
In the following we recall the classification of the spectrum of linear operator first.

\begin{definition}
\label{def3.1}
Let us denote $\sigma \left( \mathcal{L}\right) $ be the spectrum of $\mathcal{L}.$ Also denote $\sigma _{n}\left( L\right) \subset \sigma \left( \mathcal{L}\right) $ as the normal spectrum which consists of isolated eigenvalues with finite multiplicity. $\sigma _{e}\left(\mathcal{L}\right) :=\sigma \left( \mathcal{L}\right) \setminus\sigma _{n}\left( \mathcal{L}\right) $ is defined as the essential spectrum of $\mathcal{L}$.
\end{definition}

Before obtaining the spectral results of $\mathcal{L}$, in the following we prove that $\mathcal{L}$ is a generator of an analytic semigroup.

\begin{lemma}
\label{resolvent estimate}
The operator $\mathcal{L}$ defined in \eqref{L} has a dense domain. For $\re\lambda >0$ sufficiently large, the resolvent operator $(\lambda I-\mathcal{L})^{-1}$ exists and satisfying the following estimate:%
\begin{equation}
\label{L resolvent estimate}
  \left\|(\l I-\mathcal{L})^{-1}\right\|\le\frac{C}{|\l|},
\end{equation}
where $C>0$ is a constant.
\end{lemma}
\begin{proof}
See Appendix A for the detailed proof.
\end{proof}

From Lemma \ref{resolvent estimate}, we have the results about semigroup:

\begin{theorem}
\label{thm3.3}
The operator $\mathcal{L}$ defined in \eqref{L} is a generator of an analytic semigroup on $X$.
\end{theorem}
\begin{proof}
By Lemma \ref{resolvent estimate}, and the argument similar to the proof of Theorem 8.4.13 in \cite{LiWangWuYe},  one can prove that there exist $R_0\in\RR$ and a sector $S_{R_{0},\varphi }$ defined as
$$
  S_{R_0,\vp}:=\left\{\l\in\CC\Big|0\le|\arg(\l-R_0)|\le\vp,\ \vp=\frac{\pi}{2}+\tan^{-1}\Big(\frac{1}{2C}\Big),\ \l\neq R_0\right\},
$$
which is contained in the region
$$
  \left\{\l\in\CC\Big|\re\l\ge R_0,\ \l\neq R_0\right\}\cup\left(\bigcup_{r\in\RR,r\neq 0}\left\{\l\in\CC\Big||\l-(R_0+ir)|\le\frac{|r|}{2C}\right\}\right).
$$
Recall that $C$ is the constant in (\ref{L resolvent estimate}). Moreover, for all $\lambda \in S_{R_{0},\varphi },$ one can also prove that $\big\|(\l I-\mathcal{L})^{-1}\big\|\le\frac{C'}{|\l-R_0|}$ for some constant
$C'>0$. Therefore we have that $\mathcal{L}$ is a sectorial operator. Then by Theorem 8.4.18 in \cite{LiWangWuYe} or in Henry's book \cite{Henry} we have that $\mathcal{L}$ is a generator of an analytic semigroup $e^{t\mathcal{L}}$ which is defined as%
\begin{equation*}
  e^{t\mathcal{L}}:=\frac{1}{2\pi i}\int_{\Gamma }\left( \lambda I-\mathcal{L}\right) ^{-1}e^{\lambda t}d\lambda ,
\end{equation*}
where $\Gamma $ is a properly chose contour as described in \cite{Henry} or \cite{LiWangWuYe} and $i=\sqrt{-1}.$ The theorem is proved.
\end{proof}

\subsection{Instability on \boldmath$BUC(\RR;\RR^2)$ space}

In this subsection, we follow Theorem 3.3 to study the stability or instability of traveling pulse solutions by the spectral analysis, which is the extension of Theorems 1.1 and 3.1 in Chapter 5 of \cite{Volpert}. First, we establish the instability of the traveling pulse solutions of Section 2 in space $X$ which is defined in \eqref{X}.

\begin{theorem}
\label{stability with no weight}
The traveling pulse solution $(U,V)^T$ is unstable in $BUC(\RR;\RR^2)$.
\end{theorem}
\begin{proof}
Define
\begin{equation*}
\label{dispersive curvess}
  S_{\pm}:=\{\l\in\CC\big|\det(-\t^2D_{\pm}+i\t M_{\pm}+N_{\pm}-\l I_2)=0\text{ for some }\t\in\RR\},
\end{equation*}
where $D_{\pm}=D(\pm\infty),\ M_{\pm}=M(\pm\infty),\ N_{\pm}=N(\pm\infty),\ i=\sqrt{-1}$ and $I_{2}$ denotes the $2\times 2$ identity matrix.
One can see from Theorem A.2 in Chapter 5 of \cite{Henry} that $S_{\pm }$
determine the boundaries of the essential spectrum $\s_e(\mathcal{L})$. By the properties of $\left(
U,V\right) \left( \xi \right) $ as $\xi \rightarrow \pm \infty ,$ we have
$$
  D_{\pm}=\left[\begin{array}{cc}
    \e & -\chi u_{\pm}\p'(0) \\
    0  & \frac{1}{\e}
  \end{array}\right],\quad
  M_{\pm}=\left[\begin{array}{cc}
    s-\chi\p(0) & 0 \\
    0           & s
  \end{array}\right],\quad
  N_{\pm}=\left[\begin{array}{cc}
    0            & 0 \\
    \frac{1}{\e} & -\frac{1}{\e}g'(v_{\pm})
  \end{array}\right].
$$
Then $S_+$ can be verified by solving $\det(-\t^2D_++i\t M_++N_+-\l I_2)=0$ for some $\t\in\RR$. That is, $S_+$ solves
\begin{eqnarray}
0 &=& \det(-\t^2D_++i\t M_++N_+-\l I_2) \nonumber\\
  &=& \det\left[\begin{array}{cc}
        -\t^2\e+i\t(s-\chi\p(0))-\l & \chi\t^2u_+\p'(0) \\
        \frac{1}{\e}                & -\frac{\t^2}{\e}+i\t s-\frac{1}{\e}g'(v_+)-\l
      \end{array}\right] \nonumber\\
  &=& \big(\l+\t^2\e-i\t(s-\chi\p(0))\big)\Big(\l+\frac{\t^2}{\e}-i\t s+\frac{1}{\e}g'(v_+)\Big)-\frac{\chi}{\e}\t^2u_+\p'(0). \label{S_+}
\end{eqnarray}
Equation \eqref{S_+} can be written as
\begin{equation}
\label{lambda poly}
  \l^2+a(\t,\e)\l+b(\t,\e)=0,
\end{equation}
where
\begin{equation*}
  \begin{split}
    a(\t,\e)&=\Big(\e+\frac{1}{\e}\Big)\t^2-i(2s-\chi\p(0))\t+\frac{1}{\e}g'(v_+), \\
    b(\t,\e)&=\big(\t^2\e-i\t(s-\chi\p(0))\big)\Big(\frac{\t^2}{\e}-i\t s+\frac{1}{\e}g'(v_+)\Big)-\frac{\chi}{\e}\t^2u_+\p'(0).
  \end{split}
\end{equation*}
Let $\l^+_+(\t,\e)$ and $\l^-_+(\t,\e)$ denote the solutions of (\ref{lambda poly}).  Then we have
$$
  2\l_+^{\pm}(\t,\e)=-a(\t,\e)\pm\sqrt{a^2(\t,\e)-4b(\t,\e)},
$$
which implies that
$$
  2\re\l_+^{\pm}(\t,\e)=-\Big(\t^2\e+\frac{\t^2}{\e}+\frac{1}{\e}g'(v_+)\Big)\pm\re\sqrt{a^2(\t,\e)-4b(\t,\e)}.
$$
Since $g'(v_+)<0$, then $-\big(\t^2\e+\frac{\t^2}{\e}+\frac{1}{\e}g'(v_+)\big)=-\frac{1}{\e}g'(v_+)+O(1)\t^2>0$ as $\t\to 0$, for $\t$ sufficiently small we have that, depending on the signs of $\re\sqrt{a^2(\t,\e)-4b(\t,\e)}$, either $\re\l_+^+(\t,\e)>0$ or $\re\l_+^-(\t,\e)>0$. Therefore, there exists some $\l_+\in\s_e(\mathcal{L})$ with $\re\l_+>0$ for sufficiently small $\e$ and $\t$. Hence $(U,V)^T$ is linearly unstable by the results of instability in \cite{Volpert}.
\end{proof}

\subsection{Instability on weighted \boldmath$BUC(\RR;\RR^2)$ space}

From Theorem \ref{stability with no weight} we that the function space $X$ in \eqref{X} is not a proper underlying space to distinguish the stable and unstable traveling pulses. In order to extract those traveling waves (or pulse) with faster decay rate as $\xi\to\infty ,$ we normally consider the weighted function spaces \eqref{Xw} with $w=e^{\rho \xi } $ for some $\rho>0$ to be determined later. The reason is based on the
fact that the traveling pulse solution $(U,V)^T$ tends to $(u_{\pm },v_{\pm })^T$ as $\xi \rightarrow \pm \infty $ exponentially by the standard theory of geometric singular perturbations.

Define the operator $T:X_w\rightarrow X$ as
\begin{equation}
\label{T}
  T\left[\begin{array}{c}
    p \\ q
  \end{array}\right]=e^{\rho \xi }\left[\begin{array}{c}
    p \\ q
  \end{array}\right].
\end{equation}
Then, we have the operator $\mathcal{L}_{w}:X\rightarrow X$ which is defined as
\begin{equation}
\label{L_w}
  \mathcal{L}_{w}=T\mathcal{L}T^{-1}.
\end{equation}

It is easy to show that $\s(\mathcal{L}_w)=\s(\mathcal{L})$. That is, the spectrum of $\mathcal{L}$ restricted in $X_{w}$ is equal to the spectrum of $\mathcal{L}_{w}$ in $X$. Therefore, when we study the spectral problems of operator $\mathcal{L}$ in $X_w$, we can always consider the eigenvalue problem $\mathcal{L}_w\tp=\l\tp$ with $\mathbf{p}\in X$ of $\mathcal{L}_{w}$ instead of $\mathcal{L}$. From \eqref{T} and \eqref{L_w} we can rewrite $\mathcal{L}_{w}$ as
$$
  \mathcal{L}_w:=D(\xi)\partial_{\xi\xi}+M_w(\xi)\partial_{\xi}+N_w(\xi)I,
$$
where $M_w(\xi)$ and $N_w(\xi)$ defined as
$$
  M_w=M-2\r D=\left[\begin{array}{cc}
    s-\chi\p(V')-2\r\e & -\chi(U\p'(V'))'+2\r\chi U\p'(V') \\
    0                  & s-\frac{2}{\e}\r
  \end{array}\right],
$$
and
$$
  N_w=\r^2 D-\r M+N=\left[\begin{array}{cc}
    \r^2\e-s\r+\r\chi\p(V')-\chi(\p(V'))' & -\r^2\chi U\p'(V')+\r\chi(U\p'(V'))' \\
    \frac{1}{\e}                          & \frac{1}{\e}\r^2-s\r-\frac{1}{\e}g'(V)
  \end{array}\right].
$$

without loss of generality, we assume that $s>\chi\p(0)$.

\begin{theorem}
\label{stability}
The traveling pulse solution $(U,V)$ is linearly unstable in $X_{w}$ for all $\rho >0$.
\end{theorem}
\begin{proof}
As in the proof of Theorem \ref{stability with no weight}, we can obtain $\s_e(\mathcal{L}_w)$ by finding $S_{w,\pm}$ with respect to $\mathcal{L}_w$,
where
$$
 S_{w,\pm}:=\{\l\in\CC\big|\det(-\t^2D_{\pm}+i\t M_{w,\pm}+N_{w,\pm}-\l I_2)=0\text{ for some }\t\in\RR\},
$$
where $M_{w,\pm}=M_w(\pm\infty),\ N_{w,\pm}=N_w(\pm\infty)$. For the convenience, we also consider $S_{w,+}$ consisting of $\l\in\CC$ that satisfy
\begin{eqnarray}
  0 &=& \det(-\t^2D_++i\t M_{w,+}+N_{w,+}-\l I_2) \label{S_w_+}\\
    &=& R_1(\l,\t,\e)R_2(\l,\t,\e)+R_3(\l,\t,\e), \nonumber
\end{eqnarray}
where
\begin{eqnarray*}
  R_1(\l,\t,\e) &=& \l+\t^2\e-i\t(s-\chi\p(0)-2\r\e)-\r^2\e+\r(s-\chi\p(0)), \\
  R_2(\l,\t,\e) &=& \l+\frac{1}{\e}\t^2-i\t\Big(s-\frac{2}{\e}\r\Big)-\frac{1}{\e}\r^2+s\r+\frac{1}{\e}g'(v_+), \\
  R_3(\l,\t,\e) &=& -\frac{1}{\e}\big((\t^2-\r^2)\chi u_+\p'(0)+2i\t\r\chi u_+\p'(0)\big).
\end{eqnarray*}
Then \eqref{S_w_+} can be written as
\begin{equation}
\label{lambda poly weight}
  \l^2+a_w(\t,\e)\l+b_w(\t,\e)=0,
\end{equation}
where
\begin{equation*}
  \begin{split}
    a_w(\t,\e)&=\Big(\e+\frac{1}{\e}\Big)(\t^2-\r^2)+\r\big(2s-\chi\p(0)\big)+\frac{1}{\e}g'(v_+)-i\t\Big(2s-\chi\p(0)-2\r\Big(\e+\frac{1}{\e}\Big)\Big), \\
    b_w(\t,\e)&=\big(R_1(\l,\t,\e)-\l\big)\big(R_2(\l,\t,\e)-\l\big)+R_3(\l,\t,\e).
  \end{split}
\end{equation*}
Let $\l_{w,+}^+(\t,\e)$ and $\l_{w,+}^-(\t,\e)$ be the solutions of \eqref{lambda poly weight}. Then we have
$$
  2\re\l_{w,+}^{\pm}(\t,\e)=-\re\,a_w(\t,\e)\pm\re\sqrt{a_w^2(\t,\e)-4b_w(\t,\e)}.
$$
We notice that $-\re a_w(\t,\e)$ can be written as
$$
  -\re\,a_w(\t,\e)=\Big(\e+\frac{1}{\e}\Big)\r^2-\big(2s-\chi\p(0)\big)\r-\frac{1}{\e}g'(v_+)+O(1)\t^2:=\mathcal{P}(\r)+O(1)\t^2
$$
as $\t\to 0$. The discriminant of the polynomial $\mathcal{P}(\r)$, which is denoted as $\disc(\mathcal{P})$, can be calculated as
$$
  \disc(\mathcal{P})=\big(2s-\chi\p(0)\big)^2+\frac{4}{\e}\Big(\e+\frac{1}{\e}\Big)g'(v_+).
$$
Then, for sufficiently small $\e$, we have
$$
  \disc(\mathcal{P})=\frac{4}{\e^2}g'(v_+)+o\Big(\frac{1}{\e^2}\Big)<0.
$$
It follows that $\mathcal{P}(\r)>0$ for all $\r>0$. Therefore, we have $-\re\,a_w(\t,\e)>0$ for sufficiently small $\e$ and $\t$. So we obtain that either $\re\l_{w,+}^+(\t,\e)>0$ or $\re\l_{w,+}^-(\t,\e)>0$ for sufficiently small $\t$. Again, by the results of essential spectrum as in Theorem A.2 of Chapter 5 in \cite{Henry}, there exists some $\l_2\in\s_e(\mathcal{L}_w)$ satisfying $\re\l_2>0$ when $\e$, $\t$ are small enough. Therefore $(U,V)^T$ is linearly unstable in $X_w$. The proof is complete.
\end{proof}

\appendices
\appendixpage

\section{Resolvent estimates for the operator \boldmath{$\mathcal{L}$}}

In the appendix, we give the detailed proof of Lemma \ref{resolvent estimate}. From \eqref{L}, the domain $\dom(\mathcal{L})$ of operator $\mathcal{L}$ is defined as
$$
BUC^2(\RR;\RR^2):=\left\{\tp\in BUC(\RR;\RR^2)\big|\tp',\tp''\in BUC(\RR;\RR^2)\right\}.
$$

Then in order to obtain the result of Lemma \ref{resolvent estimate}, we first prove the following proposition.

\begin{prps}
\label{propa.1} $BUC^2(\RR;\RR^2)$ is dense in $BUC(\RR;\RR^2)$.
\end{prps}
\begin{proof}
Let $\Phi\in BUC(\RR;\RR^2)$. Define
$$
  \Phi_{\d}(\xi):=\int_{-\infty}^{\infty}\vp_{\d}(\xi-y)\Phi(y)dy,
$$
where $\varphi _{\delta }$ is a mollifier. Then we have that $\Phi_{\d}\in BUC^2(\RR;\RR^2)$. Also, for all $\varepsilon >0$, there exist $\delta $ such that $\|\Phi_{\d}-\Phi\|<\ve$.
\end{proof}

Define the operator $\mathcal{L}_{0}$ on $X$ as
$$
  \mathcal{L}_0:=D(\xi)\partial_{\xi\xi}.
$$

\begin{prps}
\label{propa.2}
For any $\l\in\CC$ with $\re\l>0$, the operator $\mathcal{L}_0$ has the following estimate
\begin{equation}
\label{L_0 resolvent estimate}
  \left\|(\mathcal{L}_0-\l I)^{-1}\right\|\le\frac{C_1}{|\l|},
\end{equation}
for some positive constant $C_1$.
\end{prps}
\begin{proof}
Let $\l\in\CC$ with $\re\l>0$, then we have that
\begin{equation}
\label{lambda ineq}
  \sqrt{|\l|}\le2\re\sqrt{\l}.
\end{equation}
Let $\tf=(f_1,f_2)^T\in X$. Consider the equation $(\mathcal{L}_0-\l I)\tp=\tf,\ \tp=(p,q)^T\in X$, which can be written as
\begin{equation}
\label{L_0 resolvent equation}
  \left\{\begin{split}
    \e p''-\chi U\p'(V')q''-\l p&=f_1, \\
    \tfrac{1}{\e}q''-\l q&=f_2.
  \end{split}\right.
\end{equation}
Note that the second equation of \eqref{L_0 resolvent equation} is equivalent to $q''-\e\l q=\e f_2$. Taking the Fourier transform on both side of this equation, we get
$$
  \widehat{q''}-\e\l\hat{q}=\e\hat{f}_2\Rightarrow (iy)^2\hat{q}-\e\l\hat{q}=\e\hat{f}_2\Rightarrow -(y^2+\e\l)\hat{q}=\e\hat{f}_2.
$$
It follows that
$$
  \hat{q}=-\frac{\e\hat{f}_2}{y^2+\e\l}.
$$
Define $\mathcal{H}$ be the Heaviside function as
$$
  \mathcal{H}(x)=\left\{\begin{array}{ll}
    1, & x\ge 0, \\
    0, & x<0.
   \end{array}\right.
$$
By taking the inverse Fourier transform on the above equation for $\hat{q}$ and using the residue theorem, we obtain $q(\xi)$ as follows.
\begin{align}
  q(\xi)&=-\frac{1}{\sqrt{2\pi}}\Big(\frac{\e\hat{f}_2}{y^2+\e\l}\Big)^{\widecheck{}}(\xi) \\
        &=-\frac{\e}{\sqrt{2\pi}}\Big(\frac{1}{y^2+\e\l}\Big)^{\widecheck{}}(\xi)*
           f_2(\xi)=-\frac{\e}{\sqrt{2\pi}}\Big(\frac{1}{\sqrt{2\pi}}\int_{-\infty}^{\infty}\frac{e^{i\xi y}}{y^2+\e\l}dy\Big)*f_2(\xi) \nonumber\\
        &=-\frac{\e}{2\pi}\Big[\mathcal{H}(\xi)\cdot2\pi i\res\Big(\frac{e^{i\xi z}}{z^2+\e\l};i\sqrt{\e\l}\Big)+\mathcal{H}(-\xi)\cdot2\pi i\res
           \Big(\frac{e^{i\xi z}}{z^2+\e\l};-i\sqrt{\e\l}\Big)\Big]*f_2(\xi) \nonumber\\
        &=-\frac{\e}{2\pi}\Big[\mathcal{H}(\xi)\cdot\frac{\pi e^{-\sqrt{\e\l}\xi}}{\sqrt{\e\l}}+\mathcal{H}(-\xi)\cdot\frac{\pi e^{\sqrt{\e\l}\xi}}{\sqrt{\e\l}}\Big]*f_2(\xi) \nonumber\\
        &=-\frac{\e}{2\pi}\cdot\frac{\pi}{\sqrt{\e\l}}e^{-\sqrt{\e\l}|\xi|}*f_2(\xi)=\frac{-\sqrt{\e}}{2\sqrt{\l}}\int_{-\infty}^{\infty}e^{-\sqrt{\e\l}|\xi-y|}f_2(y)dy. \label{q solution}
\end{align}
On the other hand, the first equation of the system \eqref{L_0 resolvent equation} can then be written as
\begin{eqnarray*}
  \e p''-\l p &=& \chi U\p'(V')q''+f_1, \\
              &=& \e\chi U\p'(V')(\l q+f_2)+f_1.
\end{eqnarray*}
According to \eqref{q solution}, $p(\xi)$ can be solved as
\begin{equation}
\label{p solution}
  p(\xi)=\frac{-1}{2\sqrt{\e\l}}\int_{-\infty}^{\infty}e^{-\sqrt{\frac{\l}{\e}}|\xi-y|}[\e k(y)(\l q(y)+f_2(y))+f_1(y)]dy,
\end{equation}
where $k(y):=\chi U(y)\p'(V'(y))$. Let us denote the uniform norm $\|\cdot\|_{\infty}$ by $\|\cdot\|$. Then according to \eqref{q solution} and \eqref{p solution}, we have the following estimate:
\begin{eqnarray}
  \|p\| &\le& \frac{1}{2\sqrt{\e}}\frac{1}{\sqrt{|\l|}}\|\e\l kq+\e kf_2+f_1\|\int_{-\infty}^{\infty}e^{-\re\sqrt{\l}\frac{|\xi-y|}{\sqrt{\e}}}dy, \nonumber\\
        &=&   \frac{1}{2\sqrt{\e}}\frac{1}{\sqrt{|\l|}}\|\e\l kq+\e kf_2+f_1\|\frac{2\sqrt{\e}}{\re\sqrt{\l}}. \label{p est 1} \\
  \|q\| &\le& \frac{\sqrt{\e}}{2\sqrt{|\l|}}\|f_2\|\int_{-\infty}^{\infty}e^{-\re\sqrt{\e\l}|\xi-y|}dy, \nonumber\\
        &=&   \frac{\sqrt{\e}}{2\sqrt{|\l|}}\|f_2\|\frac{2}{\sqrt{\e}\,\re\sqrt{\l}}. \nonumber
\end{eqnarray}
Using the inequality \eqref{lambda ineq}, we get
\begin{equation}
\label{q est 2}
  \|q\|\le\frac{2}{|\l|}\|f_2\|.
\end{equation}
Similarly, using \eqref{lambda ineq} and \eqref{q est 2}, the inequality \eqref{p est 1} can be refined as
\begin{eqnarray}
  \|p\| &\le& \frac{2}{|\l|}(\e|\l|\|k\|\|q\|+\e\|k\|\|f_2\|+\|f_1\|), \nonumber \\
        &\le& \frac{2}{|\l|}(3\e|\l|\|k\|\|f_2\|+\|f_1\|). \label{p est 2}
\end{eqnarray}
Recall that $\tp=(p,q)^T$ is a solution of the system \eqref{L_0 resolvent equation}, or equivalently the equation $(\mathcal{L}_0-\l I)\tp=\tf$. From \eqref{q est 2} and \eqref{p est 2}, we have that
$$
  \big\|(\mathcal{L}_0-\l I)^{-1}\tf\big\|=\|\tp\|\le\frac{2}{|\l|}\left\|\left[\begin{array}{cc}
    1 & 3\e\|k\| \\
    0 & 1
  \end{array}\right]\right\|\|\tf\|:=\frac{C_1}{|\l|}\|\tf\|.
$$
The inequality \eqref{L_0 resolvent estimate} is proved.
\end{proof}

Recall that $\mathcal{L}=\mathcal{L}_0+\mathcal{L}_1$, where $\mathcal{L}_1=M(\xi)\partial_{\xi}+N(\xi)I$. Then we have the following proposition:

\begin{prps}
\label{propa.3}
For $\re\l>0$ sufficiently large, we have
$$
\big\|\mathcal{L}_1(\mathcal{L}_0-\l I)^{-1}\big\|<1.
$$
\end{prps}
\begin{proof}
By the interpolation inequality (see, for example, \cite{Volpert}), we have that, for given $\delta >0$, $\tp\in X$ and $\a\in\RR$, there exists $\bar x\in[\a,\a+\d]$ such that
$$
  \tp'(\bar x)=\frac{\tp(\a+\d)-\tp(\a)}{\d}.
$$
Furthermore, for $x\in[\a,\a+\d]$ we have that
$$
  \tp'(x)=\tp'(\bar x)+\int_{\bar x}^x\tp''(s)ds.
$$
Hence
\begin{align*}
  |\tp'(x)| &\le |\tp'(\bar x)|+\|\tp''\|\int_{\bar x}^xds\le\frac{|\tp(\a+\d)|+|\tp(\a)|}{\d}+\d\|\tp''\|, \\
            &\le \d\|\tp''\|+\frac{2}{\d}\|\tp\|.
\end{align*}
Therefore
$$
\|\tp'\|\le\d\|\tp''\|+\frac{2}{\d}\|\tp\|.
$$
It follows that
\begin{align}
  \|\mathcal{L}_1\tp\| &= \|M\tp'+N\tp\|\le\|M\|\Big(\d\|\tp''\|+\frac{2}{\d}\|\tp\|\Big)+\|N\|\|\tp\|, \nonumber\\
                       &\le \d\|M\|\|D^{-1}\|\|\mathcal{L}_0\tp\|+\Big(\frac{2\|M\|}{\d}+\|N\|\Big)\|\tp\|. \label{L_1 estimate}
\end{align}
From Proposition \ref{propa.2} and (\ref{L_1 estimate}), we obtain
\begin{align}
  \big\|\mathcal{L}_1(\mathcal{L}_0-\l I)^{-1}\big\|
    &\le \d\|M\|\|D^{-1}\|\big\|\mathcal{L}_0(\mathcal{L}_0-\l I)^{-1}\big\|+\Big(\frac{2\|M\|}{\d}+\|N\|\Big)\big\|(\mathcal{L}_0-\l I)^{-1}\big\|, \nonumber\\
    &\le \d\|M\|\|D^{-1}\|(1+C_1)+\frac{C_1}{|\l|}\Big(\frac{2\|M\|}{\d}+\|N\|\Big). \label{L_1L_0 est}
\end{align}
For $|\l|$ sufficiently large, we have that
\begin{equation}
\label{quadr}
  |\l|^2-2C_1\big(\|N\|+4\|M\|^2\|D^{-1}\|(1+C_1)\big)|\l|+C_1^2\|N\|^2>0.
\end{equation}
This implies that
$$
  \Big(\frac{C_1\|N\|}{|\l|}-1\Big)^2-4\|M\|\|D^{-1}\|(1+C_1)\cdot\frac{2C_1\|M\|}{|\l|}>0
$$
and thus there exist $\d>0$ such that
$$
  \|M\|\|D^{-1}\|(1+C_1)\d^2+\Big(\frac{C_1\|N\|}{|\l|}-1\Big)\d+\frac{2C_1\|M\|}{|\l|}<0.
$$
Therefore, from \eqref{L_1L_0 est} and \eqref{quadr}, there exists $\d>0$ such that
$$
\big\|\mathcal{L}_1(\mathcal{L}_0-\l I)^{-1}\big\|\le\d\|M\|\|D^{-1}\|(1+C_1)+\frac{C_1}{|\l|}\Big(\frac{2\|M\|}{\d}+\|N\|\Big)<1.
$$
\end{proof}
Now, we are ready to complete the proof of Lemma 3.2.
\begin{proof}[{\bf Proof of Lemma \protect \ref{resolvent estimate}}]
The resolvent $(\mathcal{L}-\l I)^{-1}$ of $\mathcal{L}$ satisfies
$$
(\mathcal{L}-\l I)^{-1}=(\mathcal{L}_0-\l I)^{-1}\Big(I+\mathcal{L}_1(\mathcal{L}_0-\l I)^{-1}\Big)^{-1}.
$$
According to Proposition \ref{propa.3}, we have that $\big\|\mathcal{L}_1(\mathcal{L}_0-\l I)^{-1}\big\|<1$. Note that
$$
  \big(I+\mathcal{L}_1(\mathcal{L}_0-\l I)^{-1}\big)^{-1}=\sum_{\ell=0}^{\infty}\big(-\mathcal{L}_1(\mathcal{L}_0-\l I)^{-1}\big)^{\ell}.
$$
So we get
$$
  \Big\|\big(I+\mathcal{L}_1(\mathcal{L}_0-\l I)^{-1}\big)^{-1}\Big\|\le\sum_{\ell=0}^{\infty}\big\|\mathcal{L}_1(\mathcal{L}_0-\l I)^{-1}\big\|^{\ell}<\infty.
$$
Hence we have that $\big(I+\mathcal{L}_1(\mathcal{L}_0-\l I)^{-1}\big)^{-1}$ is a bounded operator. Therefore according to Proposition \ref{propa.2}, $(\mathcal{L}-\l I)^{-1}$ exists and we have the following estimate:
$$
  \big\|(\mathcal{L}-\l I)^{-1}\big\|\le \big\|(\mathcal{L}_0-\l I)^{-1}\big\|\Big\|\big(I+\mathcal{L}_1(\mathcal{L}_0-\l I)^{-1}\big)^{-1}\Big\|\le\frac{C_2}{|\l|},
$$
where $C_2=C_1\big\|\big(I+\mathcal{L}_1(\mathcal{L}_0-\l I)^{-1}\big)^{-1}\big\|$. The prove of Lemma \ref{resolvent estimate} is complete.
\end{proof}

\end{document}